%% file: arXiv.tex
\documentclass[a4paper]{article}
\usepackage[left=3.5cm,right=3.0cm,top=2.5cm,bottom=2.5cm]{geometry}
\usepackage{xcolor}
\usepackage{graphicx}
\usepackage{cite}
\usepackage{bm}
\usepackage{amsmath,amssymb,amscd,amsxtra,amsfonts}
\usepackage{lipsum}
\usepackage{amsfonts}
\usepackage{graphicx}
\usepackage{multirow}

\usepackage{array}
\usepackage{amssymb}
\usepackage{verbatim}

\usepackage{xcolor}

\usepackage{epstopdf}
\usepackage{amsmath,bm,stmaryrd}

\usepackage{amsthm}

\usepackage{mathtools}
\usepackage{indentfirst}

\newcommand\keywordsname{Key words}
\newcommand\AMSname{AMS subject classifications}

\newenvironment{@abssec}[1]
{\if@twocolumn
\section*{#1}%
\else
\vspace{.05in}\footnotesize
\parindent .2in
{\upshape\bfseries #1. }\ignorespaces
\fi}

{\if@twocolumn\else\par\vspace{.1in}\fi}

\newenvironment{keywords}{\begin{@abssec}{\keywordsname}}{\end{@abssec}}

\input{macros}

\newtheorem{theorem}{Theorem}[section]%  meant for continuous numbers
\newtheorem{lemma}[theorem]{Lemma}%

\newtheorem{remark}{Remark}%

\numberwithin{equation}{section}

\title{High-order Energy-stable and Charge-conservative Lagrangian FEM for 3D Incompressible Inductionless MHD equations with Variable Density}
\author{Lingxiao Li \thanks{School of Mathematics and Statistics, Henan University, Kaifeng 475004, China.}
\and
Bijie Song
\thanks{School of Mathematics and Statistics, Henan University, Kaifeng 475004, China.}
\and
Haiyan Su
\thanks{Academy of Mathematics and Systems Science, Xinjiang University, Urumqi 830017, China.}
\and
Jie Zhang \thanks{State Key Laboratory for Strength and Vibration of Mechanical Structures, School of Aerospace, Xi'an Jiaotong University, Xi'an 710049, China.}}
\begin{document}
\date{}
\maketitle

\begin{abstract}
  In this paper, we develop a high-order, energy-stable and
  charge-conservative Lagrangian finite element method for variable-density incompressible inductionless magnetohydrodynamic (MHD) equations.
  The method utilizes the moving high-order curved tetrahedral mesh
  to track the material interface.
  Second-order Backward Differentiation Formula (BDF2) is used for the temporal discretization of material derivative, together with the second-order Adams--Bashforth method (AB2) for the update of control points of the meshes.
  High-order isoparametric Taylor-Hood elements with grad-div stabilization are used for the velocity-pressure pair.
  While, to ensure the discrete charge conservation, high-order parametric $\BH(\Div)$-conforming element is adopted for the current density.  In the absence of external force, the unconditional energy-stability of the fully discrete scheme is proven.  Finally, 3D numerical experiments are conducted to confirm the expected high-order accuracy for
  smooth solutions, the energy stability property and the capability of the proposed method.
\end{abstract}


\section{Introduction}\label{sec:intro}

Magnetohydrodynamic (MHD) flows of electrically conducting liquids arise
in a wide range of technological applications, including liquid-metal
blankets in nuclear fusion devices \cite{kho23} and electromagnetic
processing of metals \cite{ger06}.
In some liquid-metal applications in the earth,
magnetic Reynolds number is small enough such that the magnetic field induced by the fluid motion is negligible
relative to the applied field, leading to the inductionless MHD approximation \cite{ni07a,ni07,ni12,als21}.
In heterogeneous or multi-phase mixing configurations, moving interface may be present. Moreover,
the density, electrical conductivity and dynamic viscosity may vary spatially and
exhibit jumps across material interfaces \cite{ger06,zha14mhd}.
To solve this type of problem, in this paper, we use the model of
variable-density inductionless MHD system, which couples
incompressible variable-density Navier-Stokes equations \cite{guermond2000,zha24} with
a quasi-static current density subsystem \cite{li19a,ni12}.

\par
Numerical treatments of moving material interfaces in MHD can be broadly
divided into sharp-interface and diffuse-interface formulations.
Among sharp-interface methods, front-tracking techniques represent the
interface explicitly and have been applied to free-surface MHD flows at
low magnetic Reynolds numbers \cite{sam07}.
Other sharp-interface approaches, including VOF and level-set methods \cite{ki10},
represent the moving interface through auxiliary fields on a background
mesh and have been used in multiphase and free-surface MHD simulations.
In particular, VOF-based formulations combined with adaptive mesh refinement
have been developed successfully for multi-phase liquid-metal MHD flows with large variations
in material parameters \cite{pan12,zha18,zha14,zha14mhd}.
Diffuse-interface, or phase-field formulations replace the sharp material
interface by a transition layer of finite thickness \cite{she10}.
Such formulations have also been developed for two-phase MHD and inductionless
MHD systems, with particular attention to conservation properties and
discrete energy stability \cite{dua25,sz23,wan25}.

\par
In contrast, sharp-interface strategy maintains a computational mesh
fitted to the moving interface.
In a Lagrangian formulation, the mesh moves with the material velocity, so
that the interface is transported directly with the mesh and no separate
interface-advection equation is required \cite{ben92,dob11,dob12}.
A principal limitation of purely Lagrangian methods is the deterioration of mesh quality under
large deformation.
Arbitrary Lagrangian--Eulerian (ALE) methods alleviate this limitation by
allowing the mesh velocity to differ from the fluid velocity \cite{hir74}.
In fitted interface-tracking ALE formulations, this freedom can be used to
control the interior mesh deformation while maintaining alignment with the
moving interface \cite{anj20,ger03,garcke2023}.
A recent sharp-interface ALE finite element method for two-phase ferrofluid
flows with unmatched densities was developed in \cite{wang26fhd}.
Moreover, the particle finite element method (PFEM) \cite{aub06,riz24,sal24} integrates standard finite element techniques with dynamic remeshing within a Lagrangian framework, thereby accommodating severe mesh deformation and facilitating the explicit tracking of evolving free surfaces. Furthermore, meshfree methods such as smoothed particle hydrodynamics (SPH) \cite{gar23,hop15,ros07,suc25} avoid fixed mesh-connectivity constraints. Their flexibility makes them well suited to problems involving substantial deformation and complex topological changes, such as droplet splashing, fragmentation, and coalescence. In the present work, we mainly focus on the pure Lagrangian FEM to solve the variable-density inductionless MHD equations.

In the Lagrangian methods, for problems with moving interfaces, realizing high-order spatial
accuracy generally requires the evolving geometry to be represented with comparable
accuracy (see the pioneering works in \cite{Cheng2008,dob12}).
High-order curvilinear finite element framework has therefore been
developed during past decades \cite{Cheng2008,dob11,dob12,dob13,dob14,neunteufel,nik22}.
In particular, optimal-order convergence of high-order ALE finite element methods
has recently been proved for two-phase Navier--Stokes flows \cite{limq26} without surface tension.
Beyond geometric accuracy, moving-mesh discretizations of incompressible flow and MHD systems must also account for
how to preserve physical constraints on the discrete level.
In the Eulerian framework, we refer the readers to the work in \cite{hip17,li19a,li21jcp,ni07} and references therein.
In the moving mesh or Lagrangian framework, important developments include adaptive moving mesh method for
ideal MHD using magnetic potential formulation \cite{han2007}, high-order $\BH(\Div)$-conforming ALE formulations
for fluid--structure interaction that yield exactly divergence-free
discrete velocities \cite{neunteufel}. Moreover, Nikl et al. \cite{nik22} developed a high-order curvilinear
Lagrangian MHD schemes that maintain the divergence-free constraint of magnetic field.
Garcke et al. proposed an energy stable and volume preserving ALE-FEM for
two-phase incompressible Navier-Stokes equations in \cite{garcke2023}.
In the present work, one main focus of our method is to preserve the charge-conservation precisely on the moving curved meshes, namely $\Div\BJ_h = 0$.
Charge-conservative discretizations of inductionless MHD have been developed
in both finite-volume and finite element settings.
Early finite-volume schemes were constructed on both collocated grids and staggered meshes \cite{ni07a,ni07,ni12},
which show good results at high Hartmann numbers. Within mixed finite element framework, $\BH(\Div)$-conforming
approximation of the current density was subsequently introduced for
three-dimensional inductionless MHD, yielding divergence-free discrete current density \cite{li19a}.

The charge-conservative methods mentioned above for inductionless MHD are formulated in the Eulerian framework,
where the computational mesh is fixed. The material interfaces is either represented in the diffuse manner or
reconstructed using techniques like VOF or level set.
In the Lagrangian framework \cite{Cheng2008,dob12}, however, the computational
mesh and material interface is moving with velocity, such that special care should be taken of for the finite element spaces to retain the high-order accuracy, discrete charge conservation and energy stable property.
Moreover, variable density introduces extra difficulty to realize stability, because the equation for density is first-order system. In the present work, we follow the ideas of strong mass conservation principle introduced
in \cite{dob12} to obtain an energy stable fully discrete scheme.

\par
Thus the main work in the present paper is to develop a high-order, charge-conservative and energy stable Lagrangian
finite element method for the variable-density incompressible inductionless MHD system on moving
curved meshes. The main contributions are summarized as follows:

\begin{itemize}
	\item
	We construct a high-order mixed finite element scheme on evolving curvilinear Lagrangian meshes.
    With the help of Piola's transformation \cite{fg16},
	the current density is approximated by a parametric
	$\BH(\Div)$-conforming BDM element \cite{LI2025}, while the velocity--pressure pair is
	discretized by isoparametric Taylor--Hood elements with grad-div
	stabilization \cite{or04}. As a result, our scheme can retain both high-order accuary and
	$\nabla\cdot\BJ_h=0$ exactly on the discrete level.
	
	\item	We establish the continuous energy law and construct a fully discrete scheme
    on successive Lagrangian meshes using BDF2 formulation. The discrete flow map is advanced by the second-order Adams--Bashforth
	method and the density update is based on the strong mass conservation principle in \cite{dob12}.
	Together with the cancellation of the discrete electromagnetic
	coupling terms, we prove a fully discrete energy identity
	and, in the absence of external force, unconditional energy stability.
	
	\item
	A series of 3D numerical experiments confirm the expected spatial
	convergence rates for smooth solutions, the exact discrete charge
	conservation, the energy-dissipation inequality and the applicability of the scheme.
    Especially, 3D Rayleigh-Taylor instability under a constant magnetic field is simulated.
\end{itemize}

\par
The remainder of this paper is organized as follows.
In Section \ref{sec:mathematical_model}, we present the variable-density
incompressible inductionless MHD model in the Lagrangian framework,
the mixed variational formulation, and the continuous energy law.
In Section \ref{sec:scheme}, we develop a high-order charge-conservative mixed
finite element method on evolving curvilinear meshes and establish exact
discrete charge conservation and a fully discrete energy law.
Finally, in Section \ref{sec:experiment}, we conduct a series of 3D numerical examples
to verify high-order accuracy, discrete charge conservation, energy dissipation
and the applicability of the proposed algorithm.
Concluding remarks are given in Section \ref{sec:conclusion}.

\section{Mathematical model and Lagrangian formulation}
\label{sec:mathematical_model}

\subsection{Variable-density inductionless MHD equations}
\label{subsec:global_model}
For each $t\in[0,T]$, let $\Omega(t)\subset\mathbb{R}^{3}$ be a bounded
Lipschitz domain representing the current fluid region, and set
$\Omega_0:=\Omega(0)$. We consider a global formulation of the
variable-density incompressible inductionless MHD equations on $\Omega(t)$. This formulation applies both to
interface-free inhomogeneous flows with variable density and to flows with moving
material interface.

Let $\rho_0$, $\sigma_0$, and $\nu_0$ denote the initial (or reference)
density, electrical conductivity and dynamic viscosity,
respectively. Here and below, $\widehat{\boldsymbol{x}}$ denotes a point in the
reference configuration $\Omega_0$. We assume that these material fields are essentially
bounded and uniformly positive:  there
exist constants $\rho_{\min}$, $\rho_{\max}$, $\sigma_{\min}$,
$\sigma_{\max}$, $\nu_{\min}$, and $\nu_{\max}$ such that
\begin{equation}
	\label{eq:material_bounds}
	\begin{aligned}
		0<\rho_{\min}
		&\leq \rho_0(\widehat{\boldsymbol{x}})
		\leq \rho_{\max}<\infty,\\
		0<\sigma_{\min}
		&\leq \sigma_0(\widehat{\boldsymbol{x}})
		\leq \sigma_{\max}<\infty,\\
		0<\nu_{\min}
		&\leq \nu_0(\widehat{\boldsymbol{x}})
		\leq \nu_{\max}<\infty
	\end{aligned}
	\qquad
	\text{for }\widehat{\boldsymbol{x}}\in\Omega_0.
\end{equation}

In the presence of material interface, in this paper, we use the typical two-phase flows as prototype.
Let $\Omega_{1,0}$ and $\Omega_{2,0}$ be bounded Lipschitz sub-domains satisfying
\begin{equation}
	\label{eq:domain_decomposition}
	\Omega_0\setminus\Gamma_0=\Omega_{1,0}\cup\Omega_{2,0},\qquad
	\Omega_{1,0}\cap\Omega_{2,0}=\varnothing,\qquad
	\Gamma_0=\overline{\Omega}_{1,0}\cap\overline{\Omega}_{2,0}
	=\partial\Omega_{1,0}\cap\partial\Omega_{2,0}\Subset\Omega_0.
\end{equation}
Here $\Gamma_0$ represents the initial material interface. The initial
material fields $\rho_0$, $\sigma_0$ and $\nu_0$ are assumed to be
piecewise regular with respect to this partition and may exhibit jumps
across $\Gamma_0$. In the interface-free case, we set
$\Gamma_0=\varnothing$. Figure~\ref{fig:interface1}(a) illustrates a representative initial
configuration with an internal material interface for typical two-phase flow.
In this paper, we assume that there is no interfacial mass transfer, diffusive mixing or surface tension.
\begin{figure}[!htbp]
	\centering
	\includegraphics[width=12.5cm]{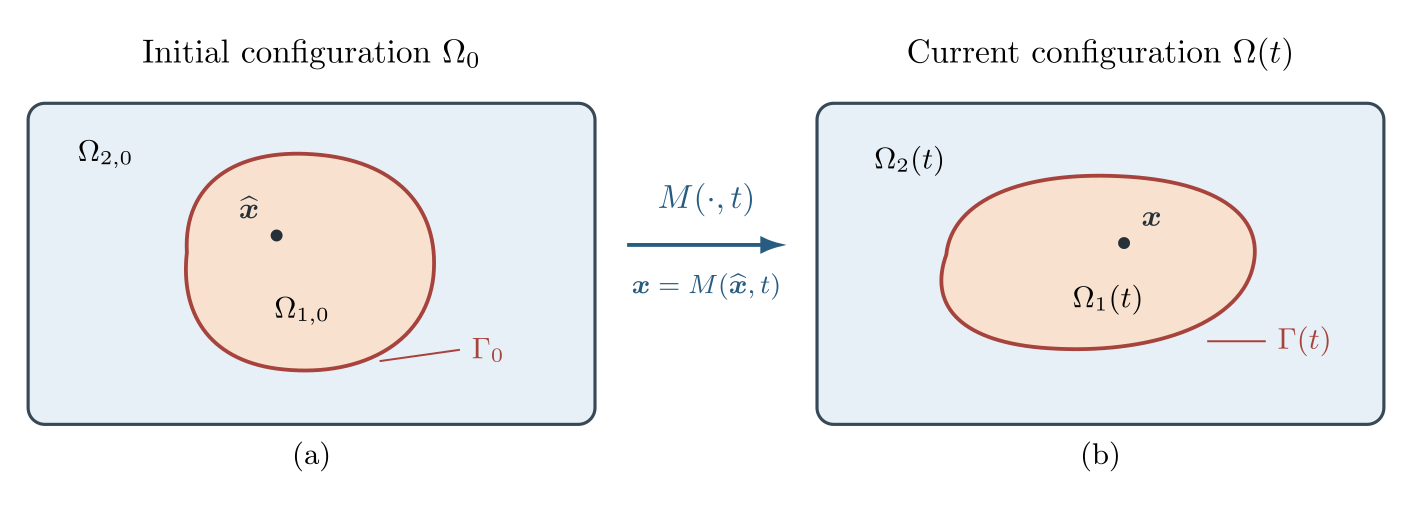}
	\caption{Representative initial and current configurations with an
		internal material interface: (a) the initial configuration
		$(\Omega_0,\Gamma_0)$; (b) the corresponding current configuration
		$(\Omega(t),\Gamma(t))$.}
	\label{fig:interface1}
\end{figure}

The unknown variables are the density $\rho$, velocity $\boldsymbol{u}$,
pressure $p$, current density $\boldsymbol{J}$ and electric
potential $\phi$. The electrical conductivity $\sigma$ and dynamic
viscosity $\nu$ are material parameters. The applied magnetic field
$\boldsymbol{B}=\boldsymbol{B}(\boldsymbol{x})$ is prescribed, while $\boldsymbol{f}_u$
denotes body force per unit mass. Under the quasistatic inductionless
approximation, the magnetic field induced by the fluid motion is
neglected relative to the applied field, and the electric field is
represented by the electrostatic potential as
$\boldsymbol{E}=-\nabla\phi$.

 The fluid motion is described by the Lagrangian flow map
$\boldsymbol{M}(\cdot,t):\Omega_0\rightarrow\Omega(t)$,
$\widehat{\boldsymbol{x}}\mapsto
\boldsymbol{x}=\boldsymbol{M}(\widehat{\boldsymbol{x}},t)$, defined by
\begin{equation}
	\label{eq:flow_map}
	\frac{\partial \boldsymbol{M}}{\partial t}
	(\widehat{\boldsymbol{x}},t)
	= \Bu(\Bx,t).
\end{equation}
For each $t\in[0,T]$, we assume that $\boldsymbol{M}(\cdot,t)$ is an
orientation-preserving diffeomorphism from $\Omega_0$ onto
$\Omega(t)$. The electrical conductivity and dynamic viscosity are evolved by
\begin{equation}
	\label{eq:material_coefficient_transport}
	\begin{aligned}
		\sigma(\Bx,t)
		=
		\sigma_0(\widehat{\boldsymbol{x}}),\quad
		\nu(\Bx,t)
		=\nu_0(\widehat{\boldsymbol{x}})
	\end{aligned}.
\end{equation}
When material interface exists, e.g. in the two-phase flow, the initial material partition is
transported by the same flow map:
\begin{equation}
	\label{eq:partition_transport}
	\Omega_i(t)=\boldsymbol{M}(\Omega_{i,0},t),
	\quad i=1,2,
	\qquad
	\Gamma(t)=\boldsymbol{M}(\Gamma_0,t).
\end{equation}
Thus $\Gamma(t)$ is a material interface transported by the fluid.
The resulting material-interface configuration is illustrated
in Figure~\ref{fig:interface1}(b).

For a sufficiently regular field $\psi(\Bx,t)$,
the material derivative is
defined by \cite{dob12}
\begin{equation}
	\label{eq:material_derivative}
	\frac{d\psi}{dt}
	:=
	\frac{\partial\psi}{\partial t}
	+
	(\boldsymbol{u}\cdot\nabla)\psi.
\end{equation}
The Cauchy stress tensor is defined by
\begin{equation}
	\label{eq:stress_tensor}
	\boldsymbol{\tau}
	:=
	-p\bbI
	+
	2\nu\boldsymbol{D}(\boldsymbol{u}),
	\qquad
	\boldsymbol{D}(\boldsymbol{u})
	:=
	\frac{1}{2}
	\left(
	\nabla\boldsymbol{u}
	+
	(\nabla\boldsymbol{u})^{T}
	\right),
\end{equation}
where $\bbI$ denotes the identity tensor. Given the symbols above, the variable-density inductionless MHD system reads
\cite{guermond2000,li19a,ni12}
\begin{subequations}
	\label{eq:model_0st}
	\begin{alignat}{2}
		\sigma^{-1}\boldsymbol{J}
		+\nabla\phi
		-\boldsymbol{u}\times\boldsymbol{B}
		&=\boldsymbol{0},
		&\quad&
		\text{in }\Omega(t),
		\label{eq:model_0st:a}\\
		\rho\frac{d\boldsymbol{u}}{dt}
		-\nabla\cdot\boldsymbol{\tau}
		-\boldsymbol{J}\times\boldsymbol{B}
		&=\rho\boldsymbol{f}_u,
		&\quad&
		\text{in }\Omega(t),
		\label{eq:model_0st:b}\\
		\frac{d\rho}{dt}
		+\rho\nabla\cdot\boldsymbol{u}
		&=0,
		&\quad&
		\text{in }\Omega(t),
		\label{eq:model_0st:c}\\
		\nabla\cdot\boldsymbol{J}
		&=0,
		&\quad&
		\text{in }\Omega(t),
		\label{eq:model_0st:d}\\
		\nabla\cdot\boldsymbol{u}
		&=0,
		&\quad&
		\text{in }\Omega(t).
		\label{eq:model_0st:e}
	\end{alignat}
\end{subequations}
For the easy presentation, we enforce the following initial and boundary conditions
\begin{subequations}
	\label{eq:initial_boundary_conditions}
	\begin{alignat}{3}
		\boldsymbol{u}(\boldsymbol{x},0)
		&=\boldsymbol{u}_0(\boldsymbol{x}),
		&\qquad
		\rho(\boldsymbol{x},0)
		&=\rho_0(\boldsymbol{x}),
		&\qquad&
		\text{in }\Omega_0,
		\label{eq:initial_boundary_conditions:a}\\
		\boldsymbol{u}
		&=\boldsymbol{0},
		&\qquad
		\phi
		&=0,
		&\qquad&
		\text{on }\partial\Omega(t).
		\label{eq:initial_boundary_conditions:b}
	\end{alignat}
\end{subequations}
Here $\Bu_0(\Bx)$ is assumed to satisfy the compatibility
conditions associated with the incompressibility constraint and the
homogeneous Dirichlet boundary condition.

\begin{remark}
	For sufficiently smooth solutions, the mass
	balance~\eqref{eq:model_0st:c} and the incompressibility
	constraint~\eqref{eq:model_0st:e} yield
	\begin{equation}
		\label{eq:density_transport}
		\frac{d\rho}{dt}=0.
	\end{equation}
	Consequently, the density remains constant along the material
	trajectories generated by the flow map $\boldsymbol{M}(\cdot,t)$. If $\rho$ has an
	interfacial jump, \eqref{eq:density_transport} holds in each
	material sub-domain where $\rho$ is smooth, whereas global mass
	conservation is expressed in the reference configuration by the
	Lagrangian identity derived in Subsection~\ref{subsec:lagrangian_weak_formulation}.
\end{remark}

We now nondimensionalize the governing system
\eqref{eq:model_0st}. To this end, let $L_c$, $u_c$, $B_c$,
$\rho_c$, $\sigma_c$ and $\nu_c$ denote characteristic scales
for length, velocity, magnetic induction, density, electrical
conductivity and dynamic viscosity, respectively. Dimensionless
quantities are denoted temporarily by a superscript asterisk.  We set
\begin{equation}
	\label{eq:dimensionless_variables}
	\begin{gathered}
		\boldsymbol{x}^{*}
		=\tfrac{\boldsymbol{x}}{L_c},
		\quad
		t^{*}
		=\tfrac{u_c}{L_c}t,
		\quad
		\rho^{*}
		=\tfrac{\rho}{\rho_c},
		\quad
		\sigma^{*}
		=\tfrac{\sigma}{\sigma_c},
		\quad
		\nu^{*}
		=\tfrac{\nu}{\nu_c},
		\\[0.6ex]
		\boldsymbol{u}^{*}
		=\tfrac{\boldsymbol{u}}{u_c},
		\quad
		p^{*}
		=\tfrac{p}{\rho_cu_c^{2}},
		\quad
		\boldsymbol{B}^{*}
		=\tfrac{\boldsymbol{B}}{B_c},
		\quad
		\boldsymbol{J}^{*}
		=\tfrac{\boldsymbol{J}}{\sigma_cu_cB_c},
		\quad
		\phi^{*}
		=\tfrac{\phi}{u_cB_cL_c},
		\quad
		\boldsymbol{f}_u^{*}
		=\tfrac{L_c}{u_c^{2}}\boldsymbol{f}_u .
	\end{gathered}
\end{equation}
The resulting dimensionless parameters are
\begin{equation}
	\label{eq:dimensionless_parameters}
	\mathrm{Re}
	=
	\frac{\rho_cu_cL_c}{\nu_c},
	\qquad
	\alpha
	=
	\frac{\sigma_cB_c^{2}L_c}{\rho_cu_c}.
\end{equation}
where $\mathrm{Re}$ and $\alpha$ denote the Reynolds number and the magnetic interaction parameter (Stuart number), respectively. Substituting \eqref{eq:dimensionless_variables} into
\eqref{eq:model_0st} and dropping the asterisks, we obtain the
non-dimensional system
\begin{subequations}
	\label{eq:model_1st}
	\begin{alignat}{2}
		\sigma^{-1}\boldsymbol{J}
		+\nabla\phi
		-\boldsymbol{u}\times\boldsymbol{B}
		&=\boldsymbol{0},
		&\quad&
		\text{in }\Omega(t),
		\label{eq:model_1st:a}\\
		\rho\frac{d\boldsymbol{u}}{dt}
		-\frac{1}{\mathrm{Re}}
		\nabla\cdot
		\left(
		2\nu\boldsymbol{D}(\boldsymbol{u})
		\right)
		+\nabla p
		-\alpha\boldsymbol{J}\times\boldsymbol{B}
		&=
		\rho\boldsymbol{f}_u,
		&\quad&
		\text{in }\Omega(t),
		\label{eq:model_1st:b}\\
		\frac{d\rho}{dt}
		+\rho\nabla\cdot\boldsymbol{u}
		&=0,
		&\quad&
		\text{in }\Omega(t),
		\label{eq:model_1st:c}\\
		\nabla\cdot\boldsymbol{J}
		&=0,
		&\quad&
		\text{in }\Omega(t),
		\label{eq:model_1st:d}\\
		\nabla\cdot\boldsymbol{u}
		&=0,
		&\quad&
		\text{in }\Omega(t).
		\label{eq:model_1st:e}
	\end{alignat}
\end{subequations}

Under this nondimensionalization, the initial and boundary conditions
retain the form stated in
\eqref{eq:initial_boundary_conditions}. Henceforth, all quantities
are understood to be dimensionless unless stated otherwise. In the
next subsection, we introduce the deformation gradient and Jacobian
determinant associated with the flow map $\BM$ and express the density $\rho$
using the strong mass conservation principle \cite{dob12}.

\subsection{Lagrangian mass conservation and global weak formulation}
\label{subsec:lagrangian_weak_formulation}

For the flow map $\BM$ defined in \eqref{eq:flow_map}, we introduce
the deformation gradient and its Jacobian determinant by
\begin{equation}
	\label{eq:deformation_jacobian}
	D\BM(\widehat{\Bx},t)
	:=
	\nabla_{\widehat{\boldsymbol{x}}}
	\BM(\widehat{\boldsymbol{x}},t),
	\qquad
	\Cj_{\BM}(\widehat{\boldsymbol{x}},t)
	:=
	\det(D\BM(\widehat{\Bx},t)).
\end{equation}
The initial condition for the flow map implies
\[D\BM(\widehat{\boldsymbol{x}},0)
=
\bbI,
\qquad
\Cj_{\BM}(\widehat{\boldsymbol{x}},0)
=1,\]
where $\bbI\in\mathbb{R}^{3\times3}$ denotes the identity matrix. For sufficiently smooth scalar- or vector-valued $\psi(\Bx,t)$, defining
\[\hat{\psi}(\widehat{\Bx},t) = \psi(\Bx,t),\]
the chain rule and \eqref{eq:flow_map} give
\begin{equation}
	\label{eq:material_lagrangian}
	\frac{\partial}{\partial t}
	\hat{\psi}\bigl(\widehat{\Bx},t\bigr)
	=
	\frac{\partial\psi}{\partial t}(\Bx,t)
	+
	\Bu\cdot\nabla\psi(\Bx,t)= \frac{d\psi}{dt}.
\end{equation}
For the derivation below, we will use the basic Liouville formula
\begin{equation}
	\label{eq:jacobian_evolution}
	\frac{\partial \Cj_{\BM}}{\partial t}
	(\widehat{\boldsymbol{x}},t)
	=  \Cj_{M}(\widehat{\Bx},t) \nabla\cdot\Bu(\Bx,t).
\end{equation}
See, for example, \cite[Eq.~(2.5)]{dob12} and \cite[Eq.~(5.18)]{ger06}.
Combining \eqref{eq:material_lagrangian},
\eqref{eq:jacobian_evolution} and the mass conservation equation
\eqref{eq:model_1st:c}, we obtain, in each material subdomain where
$\rho$ is sufficiently regular,
\begin{equation*}
\frac{d(\rho \Cj_{\BM})}{dt} = \frac{d\rho}{dt}\Cj_{\BM} + \rho \frac{\partial\Cj_{\BM}}{\partial t} =
\left[\frac{d\rho}{dt}+\rho\nabla\cdot\Bu\right]\Cj_{\BM}=0.
\end{equation*}
Integrating this identity in time, we can obtain, for $\widehat{\boldsymbol{x}}\in\Omega_0,\Bx=\BM(\widehat{\Bx},t) \in \Omega(t)$,
\begin{equation}
	\label{eq:lagrangian_mass_conservation}
	\rho(\Bx,t)
	=\frac{\rho_0(\widehat{\boldsymbol{x}})}
	{\Cj_{\BM}(\widehat{\boldsymbol{x}},t)}.
\end{equation}
Following \cite{dob12}, we refer to
\eqref{eq:lagrangian_mass_conservation} as the strong
mass-conservation principle. By the change-of-variables formula, it
implies conservation of mass in every material region transported by
$\BM$.

For the continuous incompressible problem, because
$\nabla\cdot\Bu=0$ and
$\Cj_{\BM}(\widehat{\boldsymbol{x}},0)=1$,
from \eqref{eq:jacobian_evolution}, we can conclude that $\Cj_{\BM}(\widehat{\boldsymbol{x}},t)=1$. Consequently,
\[
\rho \circ\BM = \rho_0 \quad\text{in}~~\Omega_0.
\]
Nevertheless, we will retain the Jacobian-weighted relation
\eqref{eq:lagrangian_mass_conservation}, since its discrete counterpart
provides the density update and importantly, yields the time-independent velocity mass matrix
which is crucial in the construction of energy stable fully discrete scheme.

With $\sigma$, $\nu$ and $\rho$, determined by \eqref{eq:material_coefficient_transport} and
\eqref{eq:lagrangian_mass_conservation}, we now give the mixed variational formulation on $\Omega(t)$. For each $t\in[0,T]$, define
\[
\begin{aligned}
	\boldsymbol{W}(t)
	:=
	\boldsymbol{H}
	\bigl(\operatorname{div};\Omega(t)\bigr),
	\quad
	\boldsymbol{V}(t)
:=
	\boldsymbol{H}_0^1(\Omega(t)),	\quad
	S(t)
	:=
	L^2(\Omega(t)),	\quad
	Q(t)
	:=	L_0^2(\Omega(t)).
\end{aligned}
\]
The homogeneous Dirichlet condition for the electric potential is
incorporated weakly through the boundary term arising from integration
by parts of $\nabla\phi$.

Then, for any $t\in(0,T]$, the mixed weak formulation is to
find
\[
(\BJ(t),\Bu(t),\phi(t),p(t)) \in \boldsymbol{W}(t) \times \boldsymbol{V}(t) \times S(t)\times Q(t)
\]
such that the following identities hold for all
$\boldsymbol{\varphi}\in\boldsymbol{W}(t)$,
$\boldsymbol{v}\in\boldsymbol{V}(t)$,
$r\in S(t)$ and $q\in Q(t)$:
\begin{subequations}
	\label{eq:model_2st}
	\begin{align}
		(\sigma^{-1}\BJ,\boldsymbol{\varphi})_{\Omega(t)}
		-
		(\phi,\nabla\cdot\boldsymbol{\varphi})_{\Omega(t)}
		-
		(\Bu\times\boldsymbol{B},\boldsymbol{\varphi})_{\Omega(t)}
		&=
		0,
		\label{eq:model_2st:a}\\[0.15em]
		\left(
		\rho\frac{d\Bu}{dt},
		\boldsymbol{v}
		\right)_{\Omega(t)}
		+
		\frac{1}{\mathrm{Re}}
		\left(
		2\nu\boldsymbol{D}(\Bu),
		\boldsymbol{D}(\boldsymbol{v})
		\right)_{\Omega(t)}
		-
		(p,\nabla\cdot\boldsymbol{v})_{\Omega(t)}
		-
		\alpha
		(\BJ\times\boldsymbol{B},\boldsymbol{v})_{\Omega(t)}
		&=
		(\rho\Bf_u,\boldsymbol{v})_{\Omega(t)},
		\label{eq:model_2st:b}\\[0.15em]
		-(\nabla\cdot\BJ,r)_{\Omega(t)}
		&=
		0,
		\label{eq:model_2st:c}\\
		-(\nabla\cdot\Bu,q)_{\Omega(t)}
		&=
		0.
		\label{eq:model_2st:d}
	\end{align}
\end{subequations}
Here homogeneous boundary condition for $\phi$ is incorporated weakly
in \eqref{eq:model_2st:a}, with the term arising from
integration by parts vanishing.

\begin{remark}
	The variational problem~\eqref{eq:model_2st} is posed on the
	entire domain $\Omega(t)$ and does not require a prescribed internal
	material partition. With smooth initial data, it
	describes an interface-free inhomogeneous flow. The
	constant-density, constant-coefficient problem can be seen as a
	special case. If a material interface is presented, $\rho_0,\sigma_0$ and $\nu_0$ may be
	piecewise regular and discontinuous across $\Gamma_0$; the corresponding
	density $\rho(\Bx,t)$, material parameters $\sigma(\Bx,t)$ and $\nu(\Bx,t)$ may be discontinuous across
	$\Gamma(t)=\BM(\Gamma_0,t)$.
\end{remark}

\subsection{Continuous energy law}
\label{subsec:continuous_energy_law}

We next derive the continuous energy balance for the variational
problem~\eqref{eq:model_2st}, coupled with the Lagrangian
mass-conservation relation
\eqref{eq:lagrangian_mass_conservation}. Since the applied magnetic field is prescribed,
no magnetic-field energy enters the balance; the dissipation rate consists of the viscous and Joule
contributions. We first establish the mass-weighted transport identity used to
differentiate the kinetic energy.

\begin{lemma}
	\label{lem:mass_transport}
	Assume that $\rho$ satisfies
	\eqref{eq:lagrangian_mass_conservation} and let $\psi(\Bx,t)$ be a sufficiently smooth scalar function defined on
    $\Omega(t)$. Then
	\begin{equation}
		\label{eq:mass_transport}
		\frac{\mathrm{d}}{\mathrm{d}t}
		\int_{\Omega(t)}
		\rho\psi\,\mathrm{d}\Bx
		=
		\int_{\Omega(t)}
		\rho\frac{d \psi}{dt}\,\mathrm{d}\Bx .
	\end{equation}
\end{lemma}

\begin{proof}
    Define by $\hat{\psi}(\widehat{\Bx},t) = \psi(\Bx,t)$ with $\Bx=\BM(\widehat{\boldsymbol{x}},t)$.
    Using \eqref{eq:lagrangian_mass_conservation}, we obtain
	\begin{align*}
		\int_{\Omega(t)}
		\rho(\Bx,t)\psi(\Bx,t)\,\mathrm{d}\Bx
		&=
		\int_{\Omega_0}
		\rho\bigl(
		\BM(\widehat{\boldsymbol{x}},t),t
		\bigr)
		\hat{\psi}\bigl(
		\widehat{\Bx},t
		\bigr)
		\Cj_{\BM}(\widehat{\boldsymbol{x}},t)
		\,\mathrm{d}\widehat{\boldsymbol{x}}
		\\
		&=
		\int_{\Omega_0}
		\rho_0(\widehat{\boldsymbol{x}})
		\hat{\psi}\bigl(
		\widehat{\Bx},t
		\bigr)
		\,\mathrm{d}\widehat{\boldsymbol{x}} .
	\end{align*}
	Differentiating with respect to time and applying
	\eqref{eq:material_lagrangian} give
\begin{align*}
	\frac{\mathrm{d}}{\mathrm{d}t}
	\int_{\Omega(t)}
\rho(\Bx,t)\psi(\Bx,t)\,\mathrm{d}\Bx
	&=
	\int_{\Omega_0}
	\rho_0(\widehat{\boldsymbol{x}})
	\frac{\partial}{\partial t}
	\hat{\psi}\bigl(\widehat{\Bx},t\bigr)
	\,\mathrm{d}\widehat{\boldsymbol{x}}
	\\
	&=
	\int_{\Omega_0}
	\frac{\rho_0(\widehat{\boldsymbol{x}})}{\Cj_{\BM}}\frac{\partial}{\partial t}\hat{\psi}\bigl(\widehat{\Bx},t\bigr)
	\Cj_{\BM}
	\,\mathrm{d}\widehat{\boldsymbol{x}}
	\\
	&=
	\int_{\Omega(t)}
	\rho\frac{d \psi}{d t}\,\mathrm{d}\Bx .
\end{align*}
The last equality follows from \eqref{eq:material_lagrangian} and \eqref{eq:lagrangian_mass_conservation}.
This proves \eqref{eq:mass_transport}.
\end{proof}

For $t\in[0,T]$, define the kinetic energy and the dissipation by
\begin{equation}
	\label{eq:continuous_energy_dissipation}
	\begin{aligned}
		\mathcal{E}(t)
		:=
		\frac{1}{2}
		\int_{\Omega(t)}
		\rho|\Bu|^2\,\mathrm{d}\Bx,\qquad
		\mathcal{D}(t)
		:=
		\frac{1}{\mathrm{Re}}
		\left(
		2\nu\boldsymbol{D}(\Bu),
		\boldsymbol{D}(\Bu)
		\right)_{\Omega(t)}
		+
		\alpha
		\left(
		\sigma^{-1}\BJ,
		\BJ
		\right)_{\Omega(t)} .
	\end{aligned}
\end{equation}

\begin{theorem}
	\label{thm:continuous_energy}
	Let $\rho$ be determined by
	\eqref{eq:lagrangian_mass_conservation}, and let
	$(\BJ,\Bu,\phi,p)$ be a sufficiently regular solution of
	\eqref{eq:model_2st}. Then, for  $t\in(0,T)$,
	\begin{equation}
		\label{eq:continuous_energy_law}
		\frac{d\mathcal{E}}{d t}
		+
		\mathcal{D}(t)
		=
		\left(
		\rho\Bf_u,
		\Bu
		\right)_{\Omega(t)} .
	\end{equation}
	In particular, if $\Bf_u=\boldsymbol{0}$, then
	\begin{equation}
		\label{eq:continuous_energy_identity}
		\mathcal{E}(t)
		+
		\int_0^t
		\mathcal{D}(s)\,\mathrm{d}s
		=
		\mathcal{E}(0),
		\qquad
		t\in[0,T],
	\end{equation}
	and consequently
	$\mathcal{E}(t)\leq\mathcal{E}(0)$.
\end{theorem}

\begin{proof}
	For fixed $t\in(0,T)$, choose
	$\boldsymbol{\varphi}=\BJ$ in \eqref{eq:model_2st:a} and
	$r=\phi$ in \eqref{eq:model_2st:c}. Then
	\begin{equation}\label{eq:Jxb}
	\left(
	\sigma^{-1}\BJ,
	\BJ
	\right)_{\Omega(t)}
	-
	\left(
	\Bu\times\boldsymbol{B},
	\BJ
	\right)_{\Omega(t)}
	=
	0.
	\end{equation}
	Similarly, choosing $\boldsymbol{v}=\Bu$ in
	\eqref{eq:model_2st:b} and $q=p$ in
	\eqref{eq:model_2st:d} gives
	\begin{equation}\label{eq:momentumtest}
	\left(
	\rho\frac{d\Bu}{dt},
	\Bu
	\right)_{\Omega(t)}
	+
	\frac{1}{\mathrm{Re}}
	\left(
	2\nu\boldsymbol{D}(\Bu),
	\boldsymbol{D}(\Bu)
	\right)_{\Omega(t)}
	-
	\alpha
	\left(
	\BJ\times\boldsymbol{B},
	\Bu
	\right)_{\Omega(t)}
	=
	\left(
	\rho\Bf_u,
	\Bu
	\right)_{\Omega(t)} .
	\end{equation}
	Using the scalar triple-product identity
	\[
	\left(
	\Bu\times\boldsymbol{B},
	\BJ
	\right)_{\Omega(t)}
	=
	-
	\left(
	\BJ\times\boldsymbol{B},
	\Bu
	\right)_{\Omega(t)},
	\]
	from \eqref{eq:Jxb}, we have
	\[
	-
	\left(
	\BJ\times\boldsymbol{B},
	\Bu
	\right)_{\Omega(t)}
	=
	\left(
	\sigma^{-1}\BJ,
	\BJ
	\right)_{\Omega(t)} .
	\]
	Substituting this relation into \eqref{eq:momentumtest} yields
	\[
	\left(
	\rho\frac{d\Bu}{dt},
	\Bu
	\right)_{\Omega(t)}
	+
	\frac{1}{\mathrm{Re}}
	\left(
	2\nu\boldsymbol{D}(\Bu),
	\boldsymbol{D}(\Bu)
	\right)_{\Omega(t)}
	+
	\alpha
	\left(
	\sigma^{-1}\BJ,
	\BJ
	\right)_{\Omega(t)}
	=
	\left(
	\rho\Bf_u,
	\Bu
	\right)_{\Omega(t)} .
	\]
	
	Applying Lemma~\ref{lem:mass_transport} with
	$\psi=\frac{1}{2}|\Bu|^2$ yields
	\[
	\left(\rho\frac{d\Bu}{dt},\Bu\right)_{\Omega(t)}
	= \frac{1}{2}\int_{\Omega(t)}
    \rho\left(\frac{\partial |\Bu|^2}{\partial t} + \Bu\cdot\nabla|\Bu|^2\right)\,\mathrm{d}\Bx
    =\frac{1}{2}
	\int_{\Omega(t)}
	\rho\frac{d |\Bu|^2 }{d t}\,\mathrm{d}\Bx
	=\frac{d \mathcal{E}}{d t}.
	\]
	Substitution into the preceding power identity proves
	\eqref{eq:continuous_energy_law}.	
	If $\Bf_u=\boldsymbol{0}$, integrating
	\eqref{eq:continuous_energy_law} over $(0,t)$ gives
	\eqref{eq:continuous_energy_identity}. The bounds in
	\eqref{eq:material_bounds}, together with
	\eqref{eq:material_coefficient_transport},
	\eqref{eq:lagrangian_mass_conservation}, and $\Cj_{\BM}>0$,
	imply that $\rho>0$, $\sigma>0$, and $\nu>0$ almost everywhere.
	Hence, $\mathcal{E}(t)\geq0$ and $\mathcal{D}(t)\geq0$ for
	$\mathrm{Re}>0$ and $\alpha\geq0$, which proves	$\mathcal{E}(t)\leq\mathcal{E}(0)$.
\end{proof}

\section{High-order charge-conservative mixed finite element method}\label{sec:scheme}
In this section, we develop a high-order, charge-conservative, and
unconditionally energy-stable finite element method on moving
curvilinear meshes. The material derivative of the velocity is
discretized by the second-order backward differentiation formula
(BDF2), while the discrete flow map is advanced by the explicit
second-order Adams--Bashforth (AB2) method. For the spatial
discretization, the current density is approximated by high-order
parametric BDM elements \cite{fg16,LI2025}, whereas the velocity and pressure are
approximated by high-order isoparametric Taylor--Hood elements with grad-div stabilization \cite{or04}.
Moreover, we rigorously prove the discrete energy stability of the fully discrete scheme.

\subsection{Spatial semidiscretization}
\label{subsec:semidiscrete}
Let $\Omega_{h}^0$ be an
approximation of the initial configuration $\Omega_0$, which is covered by
a fixed, conforming curvilinear tetrahedral mesh
$\mathcal{T}_{h}^0$. For a fixed integer $k\geq1$, each element
$K_0\in\mathcal{T}_{h}^0$ is the image of the reference tetrahedron
$\widehat K$ under an orientation-preserving polynomial
diffeomorphism
\[
\BF_{K_0}:
\widehat K\longrightarrow K_0,
\qquad
\BF_{K_0}
\in \BP_{k+1}(\widehat{K})\triangleq
[P_{k+1}(\widehat K)]^3.
\]
Here $P_r(\widehat K)$ denotes the space of scalar
polynomials of total degree at most $r$, for any integer $r\geq0$.
The element maps are assumed to be compatible across common faces
and uniformly regular in the standard isoparametric sense; see, e.g.,
\cite{bre08,Lenoir1986}.

For typical two-phase flow, $\Gamma_0\neq\varnothing$, let
$\Gamma_{h,0}\Subset\Omega_{h}^0$ be a discrete surface
approximating $\Gamma_0$. We assume that
$\mathcal{T}_{h}^0$ is fitted to $\Gamma_{0}$.
We further assume that $\Gamma_{h,0}$ partitions
$\Omega_{h}^0$ into two Lipschitz material subdomains satisfying
\[
\Omega_{h}^0\setminus\Gamma_{h,0}
=
\Omega_{1,h}^0\cup\Omega_{2,h}^0,
\qquad
\Omega_{1,h}^0\cap\Omega_{2,h}^0
=
\varnothing,
\qquad
\Gamma_{h,0}
=
\overline{\Omega}_{1,h,0}
\cap
\overline{\Omega}_{2,h,0}.
\]
Let $\widehat{\Bz}$ denote the local coordinate on the reference
tetrahedron $\widehat K$. For each
$K_0\in\mathcal{T}_{h,0}$, set
\[
D\boldsymbol{F}_{K_0}(\widehat{\Bz})
:=
\nabla_{\widehat{\Bz}}
\boldsymbol{F}_{K_0}(\widehat{\Bz}),
\qquad
\mathcal{J}_{K_0}(\widehat{\Bz})
\triangleq
\det(D\boldsymbol{F}_{K_0})>0.
\]

Let $\mathbf{BDM}_k(\widehat K)$ be the standard local BDM basis spaces \cite{boffibook}.
The parametric finite element spaces on the initial computational mesh $\Ct_h^0$
are defined by
\begin{align}
	\BW_{h}^0	&:=
	\left\{
	\boldsymbol{\varphi}_h
	\in	\boldsymbol{H}
	\bigl(\operatorname{div};\Omega_{h}^0\bigr)	:\left.
	\boldsymbol{\varphi}_h
	\right|_{K_0}\circ \BF_{K_0} = \frac{D\BF_{K_0}}{\Cj_{K_0}}\widehat{\boldsymbol{\varphi}},\quad
    \forall \widehat{\boldsymbol{\varphi}} \in \mathbf{BDM}_k(\widehat K),~~K_0\in\mathcal{T}_{h}^0
	\right\},	\label{eq:initial_current_density_space}\\
	\widetilde{\BV}_{h}^0	&:=
	\left\{	\boldsymbol{v}_h\in	\boldsymbol{H}^1(\Omega_{h}^0)
	:	\left.	\boldsymbol{v}_h
	\right|_{K_0}\circ\BF_{K_0}= \widehat{\Bv},\quad \forall \widehat{\Bv}\in\BP_{k+1}(\widehat K)
	,~~ K_0\in\mathcal{T}_{h}^0\right\},
    \BV_h^0 = \widetilde{\BV}_{h}^0 \cap \BH_0^1(\Omega_h^0),\label{eq:initial_velocity_space}\\
	S_{h}^0	&:=	\left\{	r_h	\in	L^2(\Omega_{h}^0):	\left.	r_h\right|_{K_0}\circ \BF_{K_0} = \widehat{r},
    \quad \forall \widehat{r}\in P_{k-1}(\widehat{K}),~~K_0\in\mathcal{T}_{h}^0\right\},
    \label{eq:initial_potential_space}	\\
	Q_{h}^0	&:=	\left\{
	q_h\in H^1(\Omega_{h}^0):\left.q_h\right|_{K_0}\circ\BF_{K_0} = \hat{q},\quad\forall \hat{q}	\in
	P_k(\widehat{K}),~~ K_0\in\mathcal{T}_{h}^0\right\}.\label{eq:initial_pressure_space}
\end{align}
The spaces $\BV_{h}^0$ and $Q_{h}^0$ form the classical isoparametric
Taylor--Hood pair of polynomial degrees $k+1$ and $k$, respectively.
The space $\BW_{h}^0$ is parametric BDM finite element space of order $k$ (more details can be seen in \cite{fg16,LI2025}).
The transformation used in \eqref{eq:initial_current_density_space} is the so-called Piola transformation.
For $t\in[0,T]$, let the semi-discrete flow map
$\BM_h(\cdot,t):\Omega_{h}^0 \longrightarrow \bbR^3 $ satisfy $\BM_h\in C^1\bigl([0,T];\widetilde{\BV}_{h}^0\bigr).$
For each $K_0 \in \mathcal{T}_{h}^0$, define
\[K(t) := \BM_h(K_0, t).\]
The current mesh and computational domain $\Omega_h(t)$ are then given by
\begin{equation}
	\label{mappsemi}
	\mathcal{T}_h(t)
	:=
	\left\{
	K(t) = \BM_h(K_0,t): K_0 \in \Ct_h^0\right\},	\quad\Omega_h(t) 	:= \cup_{K\in \Ct_h(t)} \overline{K}.
\end{equation}
Then corresponding discrete deformation gradient and Jacobian
determinant are defined by
\begin{equation}
	\label{eq:discrete_deformation_jacobian}
	D\boldsymbol{M}_h(\Bx_h^0,t)
	:=	\nabla_{\Bx_h^0}\BM_h(\Bx_h^0,t),\quad	\mathcal{J}_{\BM_h}(\Bx_h^0,t)
	:=	\det(D\boldsymbol{M}_h)	\quad	\text{for }\Bx_h^0\in\Omega_{h}^0.
\end{equation}
For each $t\in[0,T]$, define the parametric finite element spaces on $\Ct_h(t)$ by
\begin{align}
	\BW_h(t)	&:=	\left\{	\boldsymbol{\varphi}_h	\in	\boldsymbol{H}
	\bigl(\operatorname{div};\Omega_h(t)\bigr)
	:	\boldsymbol{\varphi}_h\circ\BM_h
	=	\frac{D\boldsymbol{M}_h}{\mathcal{J}_{\BM_h}}
	\boldsymbol{\varphi}_{h}^0,	\quad	\boldsymbol{\varphi}_{h}^0 \in\BW_{h}^0	\right\},\label{eq:current_current_space}
	\\
	\widetilde{\BV}_h(t)	&:=	\left\{	\boldsymbol{v}_h
	\in	\boldsymbol{H}^1(\Omega_h(t))	:	\boldsymbol{v}_h\circ\BM_h
	=	\boldsymbol{v}_{h}^0,	\quad	\boldsymbol{v}_{h}^0\in\widetilde{\BV}_{h}^0
	\right\},	\BV_h(t) = \widetilde{\BV}_h(t)\cap\BH_0^1(\Omega_h(t)),\label{eq:current_velocity_space}	\\
	S_h(t)	&:=	\left\{	r_h\in L^2(\Omega_h(t))	:	r_h\circ\BM_h	=
	r_{h}^0,\quad r_{h}^0 \in S_{h}^0	\right\},\label{eq:current_potential_space}	\\
	Q_h(t)	&:=	\left\{	q_h\in H^1(\Omega_h(t))	:	q_h\circ\BM_h	=
	q_{h}^0,\quad	q_{h}^0\in {Q}_{h}^0	\right\}.	\label{eq:current_pressure_space}
\end{align}
Let $\rho_{h}^0$, $\sigma_{h}^0$ and $\nu_{h}^0$ be given
approximations of $\rho_0$, $\sigma_0$ and $\nu_0$.
Their current values on $\Omega_h(t)$ are defined by
\begin{equation}
	\label{eq:discrete_material_transport}
	\rho_h\circ \BM_h	= \frac{\rho_{h}^0}{\Cj_{\BM_h}},
	\quad	\sigma_h\circ\BM_h =\sigma_{h}^0,
	\quad	\nu_h\circ\BM_h	= \nu_{h}^0.
\end{equation}
Since in the Taylor--Hood mixed finite element, the discrete velocity is not precisely divergence-free.
Then $\mathcal{J}_{\BM_h}$ will not be identically one.
However in the discrete case, we preserve the Jacobian factor $\Cj_{\BM_h}$
in the density update \eqref{eq:semidiscrete_kinematics}
to enforce discrete mass conservation and the energy stability.

Letting $\Bx_h \in \Omega_h(t)$, for a time-dependent discrete field $\psi_h(\Bx_h,t)$,
define the semi-discrete material derivative associated with $\Bu_h$ by
\begin{equation}
	\label{eq:semidiscrete_material_derivative}
	\frac{d^{*} \psi_h}{d t}
	:=\frac{\partial \psi_h}{\partial t} + \Bu_h\cdot\nabla\psi_h,\quad \Bu_h = \frac{\partial\BM_h}{\partial t}.
\end{equation}
Then spatially semi-discrete problem is to find the discrete flow map $\BM_h$, the density $\rho_h$ and
\[\bigl(\BJ_h,\Bu_h,\phi_h,p_h\bigr)
\in \BW_h(t)\times\BV_h(t)\times S_h(t)\times Q_h(t),\]
such that
\begin{equation}\label{eq:semidiscrete_kinematics}
\frac{\partial\BM_h}{\partial t}	=	\Bu_h, \quad \rho_h = \frac{\rho_h^0}{\Cj_{\BM_h}},
\end{equation}
and
\begin{subequations}
	\label{eq:semi_1st}
	\begin{align}
		(\sigma_h^{-1}\BJ_h,\boldsymbol{\varphi}_h)_{\Omega_h(t)}
		-
		(\phi_h,\nabla\cdot\boldsymbol{\varphi}_h)_{\Omega_h(t)}
		-
		(\Bu_h\times\boldsymbol{B},
		\boldsymbol{\varphi}_h)_{\Omega_h(t)}
		&=
		0,
		\label{eq:semi_1st:a}
		\\
		\left(\rho_h\frac{d^{*}\Bu_h}{dt},\boldsymbol{v}_h\right)_{\Omega_h(t)}
		+\frac{1}{\mathrm{Re}}\left(2\nu_h\boldsymbol{D}(\Bu_h),\boldsymbol{D}(\boldsymbol{v}_h)\right)_{\Omega_h(t)}
		+\beta(\nabla\cdot\Bu_h,\nabla\cdot\boldsymbol{v}_h)_{\Omega_h(t)} \notag\\
		\quad -
		(p_h,\nabla\cdot\boldsymbol{v}_h)_{\Omega_h(t)}	-\alpha	(\BJ_h\times\boldsymbol{B},\boldsymbol{v}_h)_{\Omega_h(t)}
		&=(\rho_h\boldsymbol{f}_u,\boldsymbol{v}_h)_{\Omega_h(t)},\label{eq:semi_1st:b}
		\\
		-(\nabla\cdot\BJ_h,r_h)_{\Omega_h(t)}	&=	0,	\label{eq:semi_1st:c}	\\
		-(\nabla\cdot\Bu_h,q_h)_{\Omega_h(t)}	&=	0,	\label{eq:semi_1st:d}
	\end{align}
\end{subequations}
holds for all
$
\bigl(\boldsymbol{\varphi}_h,\boldsymbol{v}_h,r_h,q_h\bigr)
\in\BW_h(t)\times\BV_h(t)\times S_h(t) \times Q_h(t).
$
The initial velocity is prescribed by
$\Bu_h(\cdot,0)=\Bu_{h}^0$, where
$\Bu_{h}^0\in\BV_{h}^0$ is a suitable approximation of $\Bu_0$.
Here, $\beta \sim O(1)$ is the grad--div stabilization parameter \cite{or04}.

\subsection{Fully discrete scheme and algebraic formulation}
\label{subsec:fully_discrete}
Let $t^n=n\Delta t$, $n=0,1,\ldots,N$, with $\Delta t=T/N$.
Denote by $\BM_h^n \in \widetilde{\BV}_h^0$ the fully discrete approximations
to the semi-discrete flow map $\BM_h(\Bx_h^0,t^n)$.
For $n=1,2,\ldots,N$, define the discrete mesh and domain $\Omega_h^n$ by
\[\Ct_h^n := \bigl\{ K^n = \BM_h^n(K_0): K_0\in\Ct_{h}^0\bigr\},\quad \Omega_h^n:=\cup_{K^n\in \Ct_h^n} \overline{K}^n.\]
Set
\[D\BM_h^n:=\nabla_{\Bx_h^0}\BM_h^n,\quad \mathcal{J}_{\BM_h^n}:=\det(D\BM_h^n).\]
At each time level $t^n$, let
$\BW_h^n$, $\BV_h^n$, $S_h^n$ and $Q_h^n$
denote the finite element spaces on $\Omega_h^n$ obtained from
the corresponding initial spaces by the Piola transformation and composition
mappings, which are defined by
\begin{align}
	\BW_h^n	&:=	\left\{	\boldsymbol{\varphi}_h^n	\in	\boldsymbol{H}
	\bigl(\operatorname{div};\Omega_h^n\bigr)
	:	\boldsymbol{\varphi}_h\circ\BM_h^n	=	\frac{D\boldsymbol{M}_h^n}{\mathcal{J}_{\BM_h^n}}
	\boldsymbol{\varphi}_{h}^0,	\quad	\boldsymbol{\varphi}_{h}^0 \in\BW_{h}^0	\right\},\\
	\widetilde{\BV}_h^n	&:=	\left\{	\boldsymbol{v}_h^n	\in	\boldsymbol{H}^1(\Omega_h^n):
    \boldsymbol{v}_h^n\circ\BM_h^n	=	\boldsymbol{v}_{h}^0,	\quad	\boldsymbol{v}_{h}^0\in
    \widetilde{\BV}_{h}^0\right\}, \BV_h^n = \widetilde{\BV}_h^n\cap \BH_0^1(\Omega_h^n),\\
	S_h^n	&:=	\left\{	r_h^n\in L^2(\Omega_h^n)	:	r_h^n\circ\BM_h^n	=
	r_{h}^0,\quad r_{h}^0 \in S_{h}^0	\right\},\\
	{Q}_h^n	&:=	\left\{	q_h^n\in H^1(\Omega_h^n)	:	q_h^n\circ\BM_h^n	=
	q_{h}^0,\quad	q_{h}^0\in {Q}_{h}^0	\right\}.
\end{align}
The fully discrete density $\rho_h^n$ and material parameters at $t^n$ are computed by
\begin{equation}
	\label{eq:fully_discrete_material}
	\rho_h^n\circ\BM_h^n	=\frac{\rho_{h}^0}{\Cj_{\BM_h^n}},
	\qquad
	\sigma_h^n\circ\BM_h^n
	=
	\sigma_{h}^0,
	\qquad
	\nu_h^n\circ\BM_h^n
	=
	\nu_{h}^0.
\end{equation}
Let $\Bu_h^n \in \BV_h^n$ be the fully discrete approximation of $\Bu_h(\cdot,t)$.
To advance the discrete flow map in time, we apply the second-order
Adams--Bashforth (AB2) method to the semidiscrete kinematic relation
\eqref{eq:semidiscrete_kinematics}. For $n\geq1$, this gives
\begin{equation}
	\label{eq:mesh_AB2}
	\frac{\BM_h^{n+1}-\BM_h^n}{\Delta t}
	=
	\frac{3}{2}\Bu_h^n\circ\BM_h^n
	-
	\frac{1}{2}\Bu_h^{n-1}\circ\BM_h^{n-1}.
\end{equation}
Since $\Bu_h^n\circ\BM_h^n$ and $\Bu_h^{n-1}\circ\BM_h^{n-1}$
belong to $\BV_{h}^0$,
the Adams--Bashforth combination in \eqref{eq:mesh_AB2}
is well defined in $\Omega_h^0$ and $\BV_h^0$.
With respect to a fixed basis functions of $\BV_{h}^0$,
\eqref{eq:mesh_AB2} reduces to the standard AB2 update
for the degrees of freedom of $\BM_h^{n+1}$.
The additional starting value $\Bu_h^1$ required by the two-step
AB2 scheme is obtained using the backward Euler method.

To operate velocity approximations defined on different discrete
meshes, for $0\leq m<n\leq N$, define the inter-time map
\begin{equation}
	\label{eq:intertime_map}
	\BM_h^{m,n}
	:=
	\BM_h^n\circ(\BM_h^m)^{-1}
	:	\Omega_h^m\longrightarrow\Omega_h^n,
\end{equation}
and the corresponding transformed velocity by
\begin{equation}
	\label{eq:transported_velocity}
	\widetilde{\Bu}_h^{\,n,m}
	:=	\Bu_h^m\circ(\BM_h^{m,n})^{-1}	\in\BV_h^n.
\end{equation}
The BDF2 approximation of the material derivative at $t^{n+1}$ is
then defined on $\Omega_h^{n+1}$ by
\begin{equation}
	\label{eq:BDF2_material_derivative}
	\left(\frac{\delta \Bu_h}{\delta t}\right)^{n+1}	:=	\frac{
		3\Bu_h^{n+1}-4\widetilde{\Bu}_h^{\,n+1,n}+\widetilde{\Bu}_h^{\,n+1,n-1}}{2\Delta t},
	\qquad n=1,\ldots,N-1.
\end{equation}
With $\BM_h^{n+1}$ determined by \eqref{eq:mesh_AB2}, for
$n=1,\ldots,N-1$, the fully discrete problem at $t^{n+1}$ is to find
\[
\bigl(
\BJ_h^{n+1},
\Bu_h^{n+1},
\phi_h^{n+1},
p_h^{n+1}
\bigr)
\in
\BW_h^{n+1}
\times
\BV_h^{n+1}
\times
S_h^{n+1}
\times
Q_h^{n+1}
\]
such that for all
$\boldsymbol{\varphi}_h^{n+1}\in\BW_h^{n+1}$,
$\boldsymbol{v}_h^{n+1}\in\BV_h^{n+1}$,
$r_h^{n+1}\in S_h^{n+1}$ and $q_h^{n+1}\in Q_h^{n+1}$,
\begin{subequations}
	\label{eq:fully_1st}
	\begin{align}
		&\bigl(
		(\sigma_h^{n+1})^{-1}\BJ_h^{n+1},
		\boldsymbol{\varphi}_h^{n+1}
		\bigr)_{\Omega_h^{n+1}}-
		\bigl(
		\phi_h^{n+1},
		\nabla\cdot\boldsymbol{\varphi}_h^{n+1}
		\bigr)_{\Omega_h^{n+1}}
		-
		\bigl(
		\Bu_h^{n+1}\times\boldsymbol{B},
		\boldsymbol{\varphi}_h^{n+1}
		\bigr)_{\Omega_h^{n+1}}
		=		0,		\label{eq:fully_1st:a}		\\
		&\bigl(
		\rho_h^{n+1}\left(\frac{\delta \Bu_h}{\delta t}\right)^{n+1},
		\boldsymbol{v}_h^{n+1}
		\bigr)_{\Omega_h^{n+1}}
		+
		\frac{1}{\mathrm{Re}}
		\bigl(
		2\nu_h^{n+1}\boldsymbol{D}(\Bu_h^{n+1}),
		\boldsymbol{D}(\boldsymbol{v}_h^{n+1})
		\bigr)_{\Omega_h^{n+1}}
		+		\beta
		\bigl(
		\nabla\cdot\Bu_h^{n+1},
		\nabla\cdot\boldsymbol{v}_h^{n+1}
		\bigr)_{\Omega_h^{n+1}}
		\notag\\
		&		-		\bigl(p_h^{n+1},\nabla\cdot\boldsymbol{v}_h^{n+1}\bigr)_{\Omega_h^{n+1}}
		-	\alpha	\bigl(\BJ_h^{n+1}\times\boldsymbol{B},\boldsymbol{v}_h^{n+1}\bigr)_{\Omega_h^{n+1}}
		=
		\bigl(
		\rho_h^{n+1}\boldsymbol{f}_u^{n+1},
		\boldsymbol{v}_h^{n+1}
		\bigr)_{\Omega_h^{n+1}},
		\label{eq:fully_1st:b}		\\
		-
		&\bigl(
		\nabla\cdot\BJ_h^{n+1},
		r_h^{n+1}
		\bigr)_{\Omega_h^{n+1}}
		=		0,		\label{eq:fully_1st:c}
		\\
		-
		&\bigl(
		\nabla\cdot\Bu_h^{n+1},
		q_h^{n+1}
		\bigr)_{\Omega_h^{n+1}}
		=		0,		\label{eq:fully_1st:d}
	\end{align}
\end{subequations}
where $\boldsymbol{f}_u^{n+1}:=\boldsymbol{f}_u(\cdot,t^{n+1})$.

\begin{lemma}[Exact charge conservation]
	\label{lem:fully_charge}
	For $n=1,\ldots,N-1$, the fully discrete current density
	$\BJ_h^{n+1}$ determined by \eqref{eq:fully_1st} satisfies
	\[
	\nabla\cdot\BJ_h^{n+1}=0.
	\]
\end{lemma}

\begin{proof}
	For each $K_0\in\Ct_{h}^0$, set
	\[
	K^{n+1}:=\BM_h^{n+1}(K_0),
	\quad
	\BF_{K^{n+1}}
	:=
	\BM_h^{n+1}\circ\BF_{K_0},
	\]
	and let
	$D\BF_{K^{n+1}}:=\nabla_{\widehat{\Bz}}\BF_{K^{n+1}}$ and
	$\mathcal{J}_{K^{n+1}}:=\det(D\BF_{K^{n+1}})>0$.
	By the construction of $\BW_h^{n+1}$ and the composition rule
	for the contravariant Piola transform, there exists
	$\widehat{\BJ}_{K_0}^{\,n+1}\in\mathbf{BDM}_k(\widehat K)$ such that
	\[
	\left(
	\BJ_h^{n+1}|_{K^{n+1}}
	\right)\circ\BF_{K^{n+1}}
	=
	\mathcal{J}_{K^{n+1}}^{-1}
	D\BF_{K^{n+1}}
	\widehat{\BJ}_{K_0}^{\,n+1}.
	\]
	Hence, by the Piola divergence identity
	\cite{boffibook,fg16,LI2025},
	\[
	\left(
	\nabla\cdot\BJ_h^{n+1}
	\right)\circ\BF_{K^{n+1}}
	=
	\mathcal{J}_{K^{n+1}}^{-1}
	\nabla_{\widehat{\Bz}}\cdot\widehat{\BJ}_{K_0}^{\,n+1}.
	\]
	
	Since
	\[
	\nabla_{\widehat{\Bz}}\cdot\widehat{\BJ}_{K_0}^{\,n+1}
	\in P_{k-1}(\widehat K),
	\]
	Then
	\[
	r_h^{n+1}\circ\BF_{K^{n+1}}
	=
	\nabla_{\widehat{\Bz}}\cdot\widehat{\BJ}_{K_0}^{\,n+1}
	\]
	defines a function $r_h^{n+1}\in S_h^{n+1}$.
	Using this test function in \eqref{eq:fully_1st:c} and changing
	variables elementwise, we obtain
	\[
	\sum_{K_0\in\Ct_{h}^0}
	\int_{\widehat K}
	\left|
	\nabla_{\widehat{\Bz}}\cdot\widehat{\BJ}_{K_0}^{\,n+1}
	\right|^2\,d\widehat{\Bz}	=0.
	\]
	Therefore
	\[
	\nabla_{\widehat{\Bz}}\cdot\widehat{\BJ}_{K_0}^{\,n+1}=0
	\qquad
	\text{in }\widehat K
	\]
	for every $K_0\in\Ct_{h}^0$.
	The Piola divergence identity then yields
	$\nabla\cdot\BJ_h^{n+1}=0$ in $\Omega_h^{n+1}$.
\end{proof}

For the algebraic formulation, let
$\{\boldsymbol{\varphi}_{j}^0\}_{j=1}^{N_J}$,
$\{\boldsymbol{v}_{j}^0\}_{j=1}^{N_u}$,
$\{r_{j}^0\}_{j=1}^{N_\phi}$ and $\{q_{j}^0\}_{j=1}^{N_p}$
be fixed bases of $\BW_{h}^0$, $\BV_{h}^0$, $S_{h}^0$ and $Q_h^0$,
respectively.
The corresponding basis functions on $\Ct_h^n$ satisfy
\[
\boldsymbol{\varphi}_j^n\circ\BM_h^n
=
\mathcal{J}_{\BM_h^n}^{-1}
D\BM_h^n\boldsymbol{\varphi}_{j}^0,
\qquad
\boldsymbol{v}_j^n\circ\BM_h^n
=
\boldsymbol{v}_{j}^0,
\qquad
r_j^n\circ\BM_h^n=r_{j}^0,
\qquad
q_j^n\circ\BM_h^n=q_{j}^0.
\]
By the definition of the inter-time map, we have
\begin{equation}
	\label{eq:transported_velocity_basis}
	\boldsymbol{v}_j^m\circ(\BM_h^{m,n})^{-1}
	=
	\boldsymbol{v}_j^n,
	\qquad
	0\leq m\leq n\leq N.
\end{equation}
Consequently, the density update \eqref{eq:fully_discrete_material} implies
\begin{equation}
	\label{eq:mass_matrix_invariance}
	\boldsymbol{Q}_{\Bu}^n[i,j]
	:=
	\bigl(
	\rho_h^n\boldsymbol{v}_j^n,
	\boldsymbol{v}_i^n
	\bigr)_{\Omega_h^n}
	=
	\bigl(
	\rho_{h}^0\boldsymbol{v}_{j}^0,
	\boldsymbol{v}_{i}^0
	\bigr)_{\Omega_{h}^0},
	\qquad n=0,1,2,\ldots,N .
\end{equation}
Hence the velocity mass matrix is independent of the time level.
Similarly, the Piola divergence identity yields
\begin{equation}
	\label{eq:current_constraint_invariance}
	G[i,j]
	:=
	-
	\bigl(
	r_i^n,
	\nabla\cdot\boldsymbol{\varphi}_j^n
	\bigr)_{\Omega_h^n}
	=
	-
	\bigl(
	r_{i}^0,
	\nabla\cdot\boldsymbol{\varphi}_{j}^0
	\bigr)_{\Omega_{h}^0},
	\qquad n=1,2,\cdots,N.
\end{equation}
Thus the matrix $G$ is also independent of the time level.
At $t^{n+1}$, write
\[
\BJ_h^{n+1}
=
\sum_{j=1}^{N_J}
\xi_{\BJ,j}^{n+1}\boldsymbol{\varphi}_j^{n+1},
\quad
\Bu_h^{n+1}
=
\sum_{j=1}^{N_u}
\xi_{\Bu,j}^{n+1}\boldsymbol{v}_j^{n+1},
\quad
\phi_h^{n+1}
=
\sum_{j=1}^{N_\phi}
\xi_{\phi,j}^{n+1}r_j^{n+1},
\quad
p_h^{n+1}
=
\sum_{j=1}^{N_p}
\xi_{p,j}^{n+1}q_j^{n+1}.
\]
It follows from \eqref{eq:transported_velocity_basis} that
\[
\widetilde{\Bu}_h^{\,n,n+1}
=
\sum_{j=1}^{N_u}
\xi_{\Bu,j}^{n}\boldsymbol{v}_j^{n+1},
\quad
\widetilde{\Bu}_h^{\,n-1,n+1}
=
\sum_{j=1}^{N_u}
\xi_{\Bu,j}^{n-1}\boldsymbol{v}_j^{n+1}.
\]

Define
\[
\begin{aligned}
	\boldsymbol{K}^{n+1}[i,j]
	&:=
	\bigl(
	\boldsymbol{\varphi}_j^{n+1}\times\boldsymbol{B},
	\boldsymbol{v}_i^{n+1}
	\bigr)_{\Omega_h^{n+1}},
	\\
	B^{n+1}[i,j]
	&:=
	-
	\bigl(
	q_i^{n+1},
	\nabla\cdot\boldsymbol{v}_j^{n+1}
	\bigr)_{\Omega_h^{n+1}},
	\\
	\boldsymbol{Q}_{\BJ}^{n+1}[i,j]
	&:=
	\bigl(
	(\sigma_h^{n+1})^{-1}
	\boldsymbol{\varphi}_j^{n+1},
	\boldsymbol{\varphi}_i^{n+1}
	\bigr)_{\Omega_h^{n+1}},
\end{aligned}
\]
and
\[
\begin{aligned}
	\boldsymbol{L}_{\Bu}^{n+1}[i,j]
	&:=
	\frac{1}{\mathrm{Re}}
	\bigl(
	2\nu_h^{n+1}
	\boldsymbol{D}(\boldsymbol{v}_j^{n+1}),
	\boldsymbol{D}(\boldsymbol{v}_i^{n+1})
	\bigr)_{\Omega_h^{n+1}}
	\\
	&\qquad
	+
	\beta
	\bigl(
	\nabla\cdot\boldsymbol{v}_j^{n+1},
	\nabla\cdot\boldsymbol{v}_i^{n+1}
	\bigr)_{\Omega_h^{n+1}},
	\\
	\boldsymbol{b}_{\Bu}^{n+1}[i]
	&:=
	\bigl(
	\rho_h^{n+1}\boldsymbol{f}_u^{n+1},
	\boldsymbol{v}_i^{n+1}
	\bigr)_{\Omega_h^{n+1}}.
\end{aligned}
\]
Since $\sigma_h^{n+1}>0$, $\nu_h^{n+1}>0$, and $\beta > 0$,
$\boldsymbol{Q}_{\BJ}^{n+1}$ is symmetric positive definite,
whereas $\boldsymbol{L}_{\Bu}^{n+1}$ is symmetric positive
semi-definite. Let
$\boldsymbol{\xi}_{\BJ}^{n+1}$,
$\boldsymbol{\xi}_{\Bu}^{n+1}$,
$\boldsymbol{\xi}_{\phi}^{n+1}$, and
$\boldsymbol{\xi}_p^{n+1}$
denote the corresponding coefficient vectors.
Then the fully discrete variational problem
\eqref{eq:fully_1st} is equivalent to the block algebraic system
\begin{subequations}
	\label{eq:fully_2st}
	\begin{align}
		\boldsymbol{Q}_{\BJ}^{n+1}
		\boldsymbol{\xi}_{\BJ}^{n+1}
		+
		(\boldsymbol{K}^{n+1})^T
		\boldsymbol{\xi}_{\Bu}^{n+1}
		+
		G^T\boldsymbol{\xi}_{\phi}^{n+1}
		&=
		\boldsymbol{0},
		\label{eq:fully_2st:a}
		\\
		-\alpha
		\boldsymbol{K}^{n+1}
		\boldsymbol{\xi}_{\BJ}^{n+1}
		+
		\boldsymbol{Q}_{\Bu}^0
		\frac{
			3\boldsymbol{\xi}_{\Bu}^{n+1}
			-
			4\boldsymbol{\xi}_{\Bu}^n
			+
			\boldsymbol{\xi}_{\Bu}^{n-1}
		}{2\Delta t}
		+
		\boldsymbol{L}_{\Bu}^{n+1}
		\boldsymbol{\xi}_{\Bu}^{n+1}
		+
		(B^{n+1})^T
		\boldsymbol{\xi}_p^{n+1}
		&=
		\boldsymbol{b}_{\Bu}^{n+1},
		\label{eq:fully_2st:b}
		\\
		G\boldsymbol{\xi}_{\BJ}^{n+1}
		&=
		\boldsymbol{0},
		\label{eq:fully_2st:c}
		\\
		B^{n+1}\boldsymbol{\xi}_{\Bu}^{n+1}
		&=
		\boldsymbol{0}.
		\label{eq:fully_2st:d}
	\end{align}
\end{subequations}

\subsection{Energy-stability analysis of the fully discrete scheme}
\label{subsec:energy_stability}

\begin{lemma}
	\label{lem:BDF2_algebraic_identity}
	Let $\boldsymbol Q$ be a symmetric positive definite matrix and define
	\[
	\|\boldsymbol{\xi}\|_{\boldsymbol Q}^{2}
	:=
	\boldsymbol{\xi}^{T}\boldsymbol Q\boldsymbol{\xi}.
	\]
	Then, for arbitrary vectors
	$\boldsymbol a$, $\boldsymbol b$, and $\boldsymbol c$,
	\begin{equation}
		\label{eq:BDF2_matrix_identity}
		\begin{aligned}
			2\boldsymbol a^{T}\boldsymbol Q
			(3\boldsymbol a-4\boldsymbol b+\boldsymbol c)=
			\|\boldsymbol a\|_{\boldsymbol Q}^{2}
			-
			\|\boldsymbol b\|_{\boldsymbol Q}^{2}
			+
			\|2\boldsymbol a-\boldsymbol b\|_{\boldsymbol Q}^{2}
			-
			\|2\boldsymbol b-\boldsymbol c\|_{\boldsymbol Q}^{2}
			+
			\|\boldsymbol a-2\boldsymbol b+\boldsymbol c\|_{\boldsymbol Q}^{2}.
		\end{aligned}
	\end{equation}
\end{lemma}

\begin{proof}
	Expanding the right-hand side of
	\eqref{eq:BDF2_matrix_identity} and using the symmetry of
	$\boldsymbol Q$ gives
	\[
	6\boldsymbol a^{T}\boldsymbol Q\boldsymbol a
	-
	8\boldsymbol a^{T}\boldsymbol Q\boldsymbol b
	+
	2\boldsymbol a^{T}\boldsymbol Q\boldsymbol c,
	\]
	which is precisely the left-hand side.
\end{proof}

For $n\geq1$, define the BDF2 modified discrete energy by
\begin{equation}
	\label{eq:BDF2_discrete_energy}
	\mathcal E_h^n
	:=
	\frac14
	\left(
	\|\boldsymbol{\xi}_{\Bu}^n\|_{\boldsymbol Q_{\Bu}^0}^{2}
	+
	\|2\boldsymbol{\xi}_{\Bu}^n
	-\boldsymbol{\xi}_{\Bu}^{n-1}\|_{\boldsymbol Q_{\Bu}^0}^{2}
	\right).
\end{equation}
By \eqref{eq:mass_matrix_invariance} and the transported velocity
basis, this energy admits the equivalent physical-space representation
\begin{equation}
	\label{eq:BDF2_discrete_energy_physical}
	\mathcal E_h^n
	=
	\frac14
	\int_{\Omega_h^n}
	\rho_h^n
	\left(
	|\Bu_h^n|^2
	+
	|2\Bu_h^n-\widetilde{\Bu}_h^{\,n-1,n}|^2
	\right)
	\,\mathrm d\boldsymbol x .
\end{equation}
Define the discrete dissipation by
\begin{equation}
	\label{eq:discrete_dissipation}
	\begin{aligned}
		\mathcal D_h^{n+1}
		:={}
		\frac{1}{\mathrm{Re}}
		\bigl(
		2\nu_h^{n+1}\boldsymbol D(\Bu_h^{n+1}),
		\boldsymbol D(\Bu_h^{n+1})
		\bigr)_{\Omega_h^{n+1}}+\\
		\beta
		\|\nabla\cdot\Bu_h^{n+1}\|_{L^2(\Omega_h^{n+1})}^{2}
		+
		\alpha
		\bigl(
		(\sigma_h^{n+1})^{-1}\BJ_h^{n+1},
		\BJ_h^{n+1}
		\bigr)_{\Omega_h^{n+1}} .
	\end{aligned}
\end{equation}

\begin{theorem}[Discrete energy stability]
	\label{thm:BDF2_energy_stability}
Let
	$(\BJ_h^{n+1},\Bu_h^{n+1},\phi_h^{n+1},p_h^{n+1})$
	be the solution of the fully discrete scheme
	\eqref{eq:fully_1st} for $n=1,\ldots,N-1$.
	Then
	\begin{equation}
		\label{eq:BDF2_energy_identity}
		\begin{aligned}
			\frac{\mathcal E_h^{n+1}-\mathcal E_h^n}{\Delta t}
			&+
			\frac{1}{4\Delta t}
			\left\|
			\boldsymbol{\xi}_{\Bu}^{n+1}
			-
			2\boldsymbol{\xi}_{\Bu}^{n}
			+
			\boldsymbol{\xi}_{\Bu}^{n-1}
			\right\|_{\boldsymbol Q_{\Bu}^0}^{2}
			+
			\mathcal D_h^{n+1}=
			\bigl(
			\rho_h^{n+1}\boldsymbol f_\Bu^{n+1},
			\Bu_h^{n+1}
			\bigr)_{\Omega_h^{n+1}} .
		\end{aligned}
	\end{equation}
	In particular, if $\boldsymbol f_u=\boldsymbol0$, then
	\begin{equation}
		\label{eq:BDF2_energy_inequality}
		\frac{\mathcal E_h^{n+1}-\mathcal E_h^n}{\Delta t}
		+
		\mathcal D_h^{n+1}
		\leq0,
	\end{equation}
	and hence
	\[
	\mathcal E_h^{n+1}\leq\mathcal E_h^n .
	\]
\end{theorem}

\begin{proof}
	Left-multiply
	\eqref{eq:fully_2st:a}--\eqref{eq:fully_2st:d}
	by
	$\alpha(\boldsymbol{\xi}_{\BJ}^{n+1})^T$,
	$(\boldsymbol{\xi}_{\Bu}^{n+1})^T$,
	$-\alpha(\boldsymbol{\xi}_{\phi}^{n+1})^T$, and
	$-(\boldsymbol{\xi}_{p}^{n+1})^T$, respectively, and sum the
	resulting identities.
	Since
	\[
	(\boldsymbol{\xi}_{\BJ}^{n+1})^T
	(\boldsymbol K^{n+1})^T
	\boldsymbol{\xi}_{\Bu}^{n+1}
	=
	(\boldsymbol{\xi}_{\Bu}^{n+1})^T
	\boldsymbol K^{n+1}
	\boldsymbol{\xi}_{\BJ}^{n+1},
	\]
	and similarly for the $G$- and $B^{n+1}$-blocks, the coupling and
	constraint terms cancel. We obtain
	\begin{equation}
		\label{eq:combined_energy_BDF2}
		\begin{aligned}
			&
			(\boldsymbol{\xi}_{\Bu}^{n+1})^T
			\boldsymbol Q_{\Bu}^{0}
			\frac{
				3\boldsymbol{\xi}_{\Bu}^{n+1}
				-
				4\boldsymbol{\xi}_{\Bu}^{n}
				+
				\boldsymbol{\xi}_{\Bu}^{n-1}}
			{2\Delta t}
			+
			(\boldsymbol{\xi}_{\Bu}^{n+1})^T
			\boldsymbol L_{\Bu}^{n+1}
			\boldsymbol{\xi}_{\Bu}^{n+1}
			+\\
			\alpha
			(\boldsymbol{\xi}_{\BJ}^{n+1})^T
			\boldsymbol Q_{\BJ}^{n+1}
			\boldsymbol{\xi}_{\BJ}^{n+1}=
			(\boldsymbol{\xi}_{\Bu}^{n+1})^T
			\boldsymbol b_{\Bu}^{n+1}.
		\end{aligned}
	\end{equation}
	
	Applying Lemma~\ref{lem:BDF2_algebraic_identity} with
	\[
	\boldsymbol a=\boldsymbol{\xi}_{\Bu}^{n+1},
	\qquad
	\boldsymbol b=\boldsymbol{\xi}_{\Bu}^{n},
	\qquad
	\boldsymbol c=\boldsymbol{\xi}_{\Bu}^{n-1},
	\qquad
	\boldsymbol Q=\boldsymbol Q_{\Bu}^{0},
	\]
	gives
	\begin{equation}
		\label{eq:BDF2_time_energy}
		\begin{aligned}
			(\boldsymbol{\xi}_{\Bu}^{n+1})^T
			\boldsymbol Q_{\Bu}^{0}
			\frac{
				3\boldsymbol{\xi}_{\Bu}^{n+1}
				-
				4\boldsymbol{\xi}_{\Bu}^{n}
				+
				\boldsymbol{\xi}_{\Bu}^{n-1}}
			{2\Delta t}=
			\frac{\mathcal E_h^{n+1}-\mathcal E_h^n}{\Delta t}
			+
			\frac{1}{4\Delta t}
			\left\|
			\boldsymbol{\xi}_{\Bu}^{n+1}
			-
			2\boldsymbol{\xi}_{\Bu}^{n}
			+
			\boldsymbol{\xi}_{\Bu}^{n-1}
			\right\|_{\boldsymbol Q_{\Bu}^{0}}^{2}.
		\end{aligned}
	\end{equation}
	Furthermore, by the definitions of
	$\boldsymbol L_{\Bu}^{n+1}$,
	$\boldsymbol Q_{\BJ}^{n+1}$, and
	$\boldsymbol b_{\Bu}^{n+1}$,
	\[
	(\boldsymbol{\xi}_{\Bu}^{n+1})^T
	\boldsymbol L_{\Bu}^{n+1}
	\boldsymbol{\xi}_{\Bu}^{n+1}
	+
	\alpha
	(\boldsymbol{\xi}_{\BJ}^{n+1})^T
	\boldsymbol Q_{\BJ}^{n+1}
	\boldsymbol{\xi}_{\BJ}^{n+1}
	=
	\mathcal D_h^{n+1},
	\]
	and
	\[
	(\boldsymbol{\xi}_{\Bu}^{n+1})^T
	\boldsymbol b_{\Bu}^{n+1}
	=
	\bigl(
	\rho_h^{n+1}\boldsymbol f_u^{n+1},
	\Bu_h^{n+1}
	\bigr)_{\Omega_h^{n+1}}.
	\]
	Substituting these identities and
	\eqref{eq:BDF2_time_energy} into
	\eqref{eq:combined_energy_BDF2}
	yields \eqref{eq:BDF2_energy_identity}.
	If $\boldsymbol f_u=\boldsymbol0$, dropping the nonnegative BDF2
	remainder gives \eqref{eq:BDF2_energy_inequality}, and hence
	$\mathcal E_h^{n+1}\leq\mathcal E_h^n$.
\end{proof}

\begin{remark}
	The discrete energy identity \eqref{eq:BDF2_energy_identity}, and hence
	the corresponding energy-stability result, remains valid when the
	homogeneous Dirichlet condition on the velocity is replaced by the
	homogeneous traction condition
	\[
	\left(
	-p\bbI
	+\frac{2\nu}{\mathrm{Re}}\boldsymbol{D}(\Bu)
	\right)\Bn
	=\boldsymbol{0}
	\qquad \text{on } \partial\Omega(t).
	\]
	Indeed, for any $\Bv\in\boldsymbol{H}^1(\Omega(t))$, integration by
	parts gives
	\[
	\begin{aligned}
		\left(-\frac{1}{\mathrm{Re}}
		\nabla\cdot\bigl(2\nu\boldsymbol{D}(\Bu)\bigr)
		+\nabla p,\Bv
		\right)_{\Omega(t)}=\\
		\frac{1}{\mathrm{Re}}
		\bigl(
		2\nu\boldsymbol{D}(\Bu),
		\boldsymbol{D}(\Bv)
		\bigr)_{\Omega(t)}
		-
		\bigl(
		p,\nabla\cdot\Bv
		\bigr)_{\Omega(t)}
		-
		\left\langle
		\left(
		-p\bbI
		+\frac{2\nu}{\mathrm{Re}}\boldsymbol{D}(\Bu)
		\right)\Bn,
		\Bv
		\right\rangle_{\partial\Omega(t)} .
	\end{aligned}
	\]
	Hence, the boundary term vanishes under the homogeneous traction
	condition, while the remaining terms have the same form as those in
	the preceding energy argument.
\end{remark}

\section{Numerical experiments}\label{sec:experiment}
In this section, we provide numerical examples to verify the desired properties of our proposed scheme.
In Example~\ref{ex:accuracy}, we verified the computational accuracy and charge conservation property
of the Lagrangian finite element algorithm.
Example~\ref{ex:energy} verified the energy stability of the fully discrete scheme \eqref{eq:fully_1st}.
Then Example~\ref{ex:application} simulates a magnetic Rayleigh-Taylor instability.
The finite element code is developed based on the PHG library \cite{zha22}.
The 3D figures are generated by the open-source software ParaView.

\subsection{Spatial convergence and charge conservation}\label{ex:accuracy}
We consider the following smooth manufactured solution:
\[\left\{
\begin{aligned}
	\rho &= 1,\\
	\BB &= (0,1,0)^\top,\\
	\BJ &= \bigl(\sin(y+z),\,\cos(\pi z),\,\sin(x+y)\bigr)^\top,\\
	\phi &= \sin(z),\\
	\Bu &= \bigl(\sin(\pi x)\cos(\pi y),\,
	-\cos(\pi x)\sin(\pi y),\,0\bigr)^\top,\\
	p &= \frac14\bigl(\cos(2\pi x)+\cos(2\pi y)\bigr).
\end{aligned}
\right.\]
For this test, the source terms, initial data
and boundary data are prescribed consistently with the above solution.
The initial computational domain is unit cube $\Omega_0=[0,1]^3$. Since exact velocity satisfies
$\Bu\cdot\Bn=0$ on $\partial\Omega_0$, the boundary has zero normal
velocity and the computational domain remains fixed. Moreover we set
\[T=0.1,\quad \Delta t=10^{-3},\quad \mathrm{Re}=\alpha=\sigma=\nu=1,\quad \beta=0.25.\]
The coupled linear systems in \eqref{eq:fully_2st} are solved by
preconditioned FGMRES with a relative tolerance of $10^{-8}$.

We firstly take $k=1$. Quadratic Lagrange elements are used for the
discrete flow map, the isoparametric $\BP_2$--$P_1$ Taylor--Hood pair
for the velocity and pressure, the first order parametric BDM element \cite{fg16,LI2025}
for the current density, and discontinuous $P_0$ elements for the
electric potential.
The expected spatial convergence rates are
\begin{equation}
	\begin{aligned}
		\left\|\Bu-\Bu_h^N\right\|_{1,\Omega_h^N}
		&=\mathcal{O}(h^2),\quad
		\left\|p-p_h^N\right\|_{0,\Omega_h^N}
		&=\mathcal{O}(h^2),\quad
		\left\|\BJ-\BJ_h^N\right\|_{\operatorname{div},\Omega_h^N}
		&=\mathcal{O}(h^2).
	\end{aligned}
	\label{eq:expected_rate_k1}
\end{equation}

The errors and observed convergence rates reported in
Table~\ref{tab:conv_k1}. The numerical results  are consistent with the expected second-order convergence as the mesh is refined.
 Moreover,
$\|\nabla\cdot\BJ_h^N\|_{0,\Omega_h^N}$ is very small, in agreement with the discrete charge
conservation established in Lemma~\ref{lem:fully_charge}.
Although the computational domain remains fixed, the computational mesh
deforms with the flow. Figure~\ref{fig:curve_mesh2s} shows the
deformed curvilinear mesh and the computed velocity field at
$T=0.1$ for $h=0.2165$.
\begin{table}[htb]
	\centering
	\fontsize{9pt}{12pt}\selectfont
	\caption{Errors and observed convergence rates for $k=1$
		(Example~\ref{ex:accuracy}).}
	\vspace{0.1in}
	\label{tab:conv_k1}
	\begin{tabular}{@{\hspace{0.5em}}c@{\hspace{0.5em}}
			c@{\hspace{0.5em}}
			c@{\hspace{0.5em}}
			c@{\hspace{0.5em}}
			c@{\hspace{0.5em}}
			c@{\hspace{0.5em}}
			c@{\hspace{0.5em}}
			c@{\hspace{0.5em}}}
		\hline
		$h$ &
		$\left\|\Bu-\Bu_h^N\right\|_{1,\Omega_h^N}$ &
		rate &
		$\left\|p-p_h^N\right\|_{0,\Omega_h^N}$ &
		rate &
		$\left\|\BJ-\BJ_h^N\right\|_{\operatorname{div},\Omega_h^N}$ &
		rate &
		$\left\|\nabla\cdot\BJ_h^N\right\|_{0,\Omega_h^N}$
		\\
		\hline
		0.8660 & 4.5070e-01 & -      & 8.3942e-01 & -      & 5.8092e-02 & -      & 6.3658e-11 \\
		0.4330 & 1.8913e-01 & 1.2528 & 1.1433e-01 & 2.8762 & 1.4562e-02 & 1.9961 & 6.9094e-12 \\
		0.2165 & 4.9288e-02 & 1.9401 & 3.3089e-02 & 1.7888 & 3.6211e-03 & 2.0077 & 1.1645e-11 \\
		0.1083 & 1.2320e-02 & 2.0002 & 6.7984e-03 & 2.2831 & 9.0384e-04 & 2.0023 & 3.3687e-11 \\
		\hline
	\end{tabular}
\end{table}

\begin{figure}[!htbp]
	\centering
	\includegraphics[width=10cm]{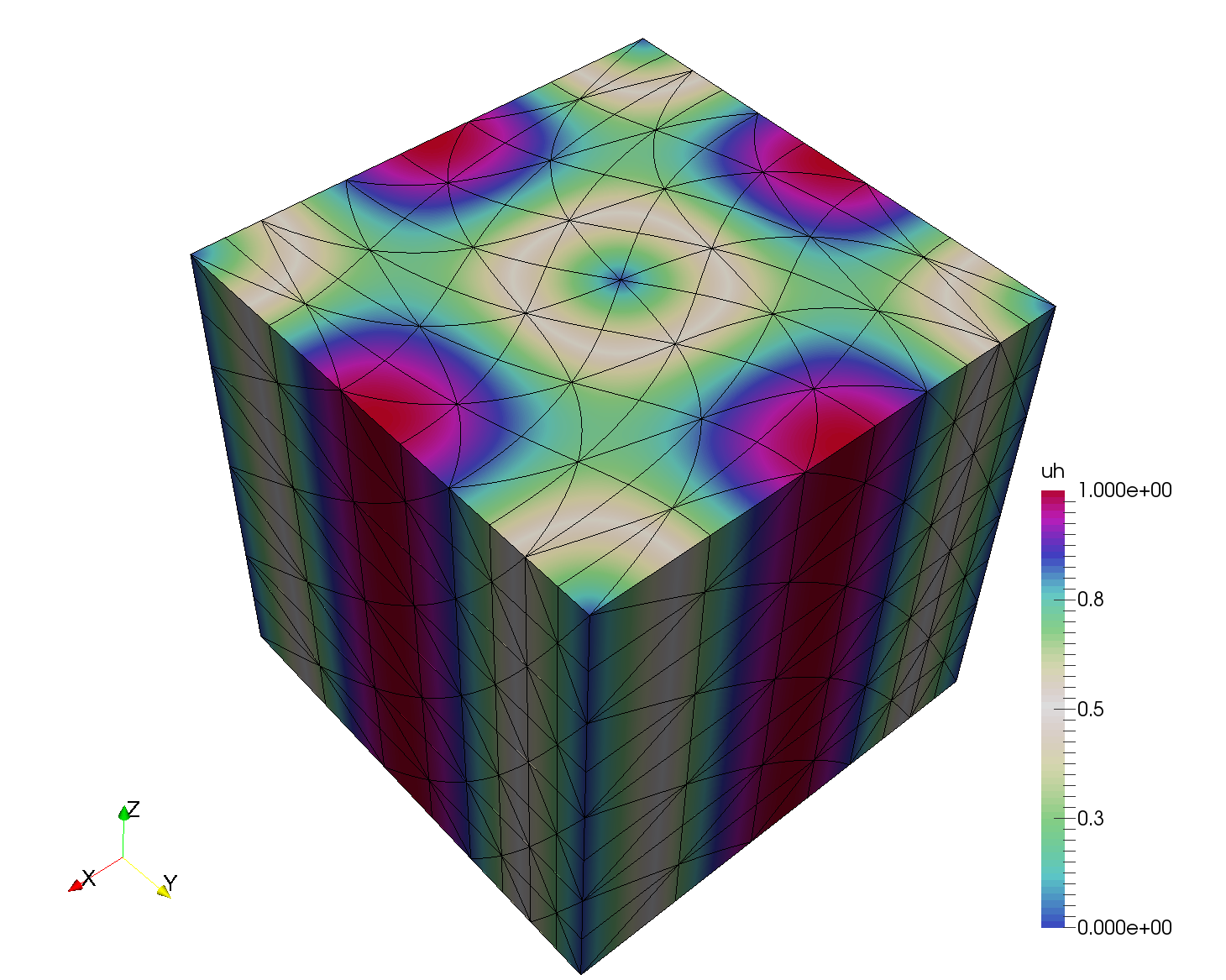}
	\caption{Deformed curvilinear mesh and computed velocity field at
			$T=0.1$ for $h=0.2165$ 	(Example~\ref{ex:accuracy}).}
	\label{fig:curve_mesh2s}
\end{figure}

We next take $k=2$ to assess the higher-order spatial accuracy of the
method. Cubic Lagrange elements are used for the discrete flow map,
the isoparametric $\BP_3$--$P_2$ Taylor--Hood pair for the velocity and
pressure, the second order parametric BDM element \cite{fg16,LI2025} for the current
density, and discontinuous $P_1$ elements for the electric potential.
The expected spatial convergence rates are
\begin{equation}
	\begin{aligned}
		\left\|\Bu-\Bu_h^N\right\|_{1,\Omega_h^N}
		&=\mathcal{O}(h^3),\quad
		\left\|p-p_h^N\right\|_{0,\Omega_h^N}
		&=\mathcal{O}(h^3),\quad
		\left\|\BJ-\BJ_h^N\right\|_{\operatorname{div},\Omega_h^N}
		&=\mathcal{O}(h^3).
	\end{aligned}
	\label{eq:expected_rate_k2}
\end{equation}

\begin{table}[htb]
	\centering
	\fontsize{9pt}{12pt}\selectfont
	\caption{Errors and observed convergence rates for $k=2$
		(Example~\ref{ex:accuracy}).}
	\vspace{0.1in}
	\label{tab:conv_k2}
	\begin{tabular}{@{\hspace{0.5em}}c@{\hspace{0.5em}}
			c@{\hspace{0.5em}}
			c@{\hspace{0.5em}}
			c@{\hspace{0.5em}}
			c@{\hspace{0.5em}}
			c@{\hspace{0.5em}}
			c@{\hspace{0.5em}}
			c@{\hspace{0.5em}}}
		\hline
		$h$ &
		$\left\|\Bu-\Bu_h^N\right\|_{1,\Omega_h^N}$ &
		rate &
		$\left\|p-p_h^N\right\|_{0,\Omega_h^N}$ &
		rate &
		$\left\|\BJ-\BJ_h^N\right\|_{\operatorname{div},\Omega_h^N}$ &
		rate &
		$\left\|\nabla\cdot\BJ_h^N\right\|_{0,\Omega_h^N}$
		\\
		\hline
		0.8660 & 2.4908e-01 & -      & 5.5357e-01 & -      & 1.0338e-02 & -      & 6.2713e-10 \\
		0.4330 & 2.1975e-02 & 3.5027 & 2.0328e-02 & 4.7672 & 1.3861e-03 & 2.8989 & 7.6111e-11 \\
		0.2165 & 2.5923e-03 & 3.0836 & 1.3703e-03 & 3.8909 & 1.7414e-04 & 2.9927 & 2.0447e-11 \\
		0.1083 & 3.1705e-04 & 3.0315 & 1.3791e-04 & 3.3127 & 2.1843e-05 & 2.9950 & 4.9237e-11 \\
		\hline
	\end{tabular}
\end{table}
The numerical results reported in Table~\ref{tab:conv_k2} show that the observed
convergence rates approach the expected third order as the mesh is
refined, while
$\|\nabla\cdot\BJ_h^N\|_{0,\Omega_h^N}$ remains very small, again in agreement with the discrete charge
conservation in Lemma~\ref{lem:fully_charge}. Overall, these computations verify the expected high-order accuracy and discrete charge conservation for our scheme.

\subsection{Energy stability verification}\label{ex:energy}
To verify the discrete energy-dissipation inequality, we consider a
variable-density problem with discontinuous electrical conductivity on
the initial domain $\Omega_0=[0,1]^3$. The initial density
and velocity are prescribed by
\[
\rho_0=1+0.3\sin(\pi z),\qquad
\Bu_0=
\bigl(
\sin(\pi x)\cos(\pi y),\,
-\cos(\pi x)\sin(\pi y),\,0
\bigr)^\top .
\]
The body force $\Bf_u$ is set to be zero and
the applied magnetic field is $\BB=(0,0,1)^\top$. For the velocity,
we impose the homogeneous traction condition $\left(-p\boldsymbol I+\frac{2\nu}{\mathrm{Re}}\BD(\Bu)
\right)\Bn=\boldsymbol 0$ $\text{on }\partial\Omega(t)$,
together with boundary condition $\phi=0$ for electric potential.
As discussed in Remark~3, the discrete energy estimate remains valid.
The dimensionless parameters are chosen as
\[
\mathrm{Re}=1000,\qquad
\alpha=\nu_0=1,\qquad
\beta=0.25.
\]
The electrical conductivity is prescribed in the reference
configuration by
\begin{equation}
	\sigma_0(x,y,z)=
	\begin{cases}
		10^{-3}, & 0.5\le z\le1,\\
		1,       & 0\le z<0.5,
	\end{cases}
	\qquad (x,y,z)\in\Omega_0 .
	\label{eq:conductivity_energy}
\end{equation}
Its discrete approximation is transported with the Lagrangian flow
according to the material-transport relation defined in \eqref{eq:fully_discrete_material}.
For the numerical discretization, we take $k=1$,
$h=0.2165$, $\Delta t=0.01$ and $T=1$. The coupled linear systems in
\eqref{eq:fully_2st} are solved by FGMRES with a relative tolerance of $10^{-8}$.
For this example, the mesh has 3,072 elements.
The velocity has 14,739 degrees of freedoms (DOFs) and $\BJ_h$ has 19,584 DOFs.
For $n\geq1$, define the discrete energy increment rate together with the dissipation terms by
\begin{equation}
	\begin{aligned}
		\left(\frac{dE}{dt}\right)^{n+1}
		&:=\frac{\mathcal{E}_h^{n+1}-\mathcal{E}_h^n}{\Delta t},\\
		\mathrm{DE}^{n+1}		&:=	\frac{1}{\mathrm{Re}}	\left(2\nu_h^{n+1}\BD(\Bu_h^{n+1}),
		\BD(\Bu_h^{n+1})\right)_{\Omega_h^{n+1}}
		+\beta\left\|
		\nabla\cdot\Bu_h^{n+1}
		\right\|_{0,\Omega_h^{n+1}}^2,\\
		\mathrm{Ohm}^{n+1}	&:=	\alpha\left((\sigma_h^{n+1})^{-1}\BJ_h^{n+1},\BJ_h^{n+1}\right)_{\Omega_h^{n+1}} .
	\end{aligned}\label{eq:energy_dissipation_terms}
\end{equation}
where $\mathcal{E}_h^n$ denotes the modified BDF2 energy defined in \eqref{eq:BDF2_discrete_energy_physical}.
Since $\boldsymbol f_u=\boldsymbol 0$,
Theorem~\ref{thm:BDF2_energy_stability} together with Remark~3, yields
\begin{equation}
	\left(\frac{dE}{dt}\right)^{n+1}+\mathrm{DE}^{n+1}+ \mathrm{Ohm}^{n+1}	\leq 0,	\quad n\geq1.
	\label{eq:energy_dissipation_test}
\end{equation}
Figure~\ref{fig:energy_trend} shows
$\left(\frac{dE}{dt}\right)^{n+1}$,$\mathrm{DE}^{n+1}$ and $\mathrm{Ohm}^{n+1}$ in the left panel and their sum in the
right panel. The sum remains non-positive at all time levels,
in agreement with the energy dissipation inequality \eqref{eq:energy_dissipation_test}.
\begin{figure}[!htbp]
	\centering
	\includegraphics[width=7cm]{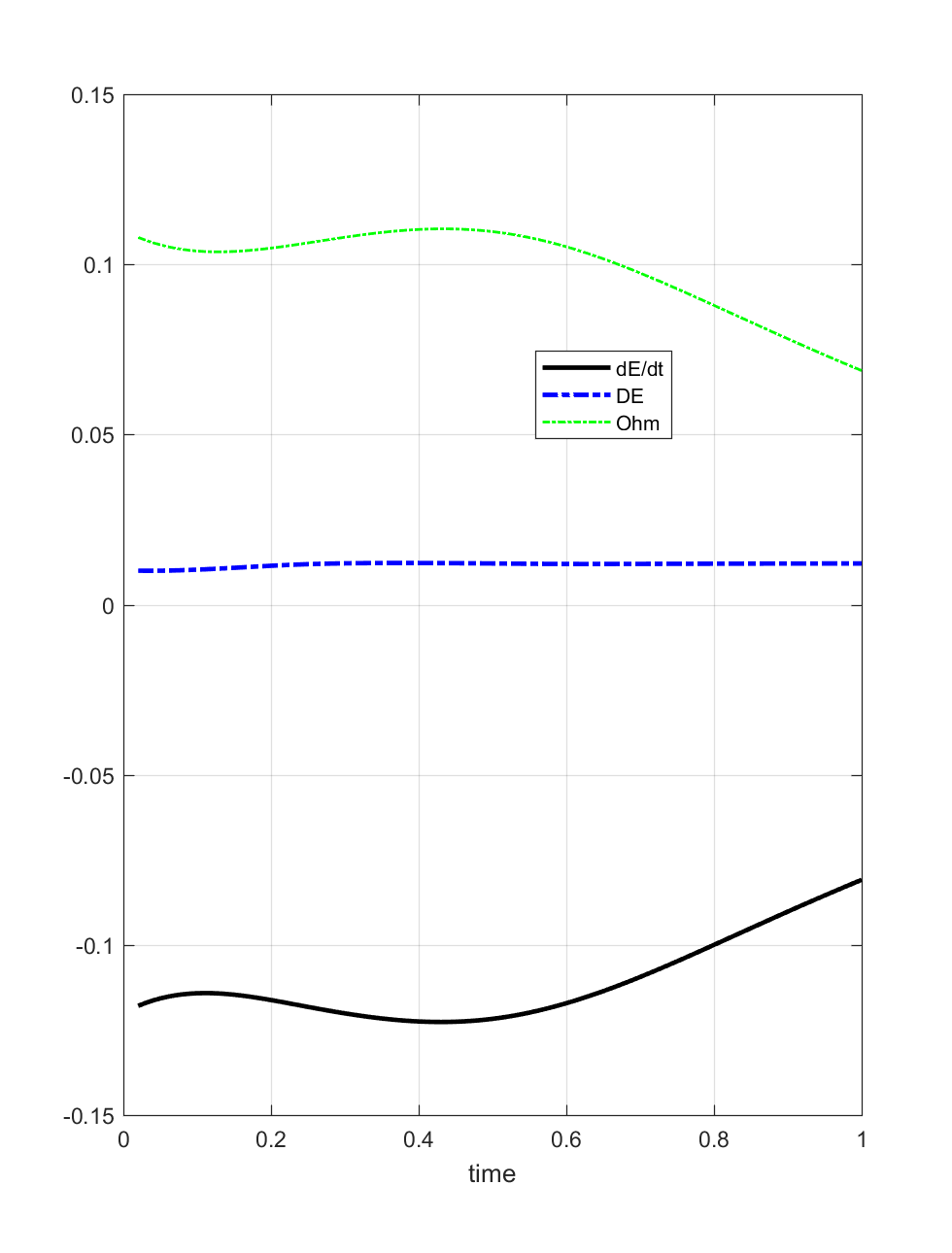}
	\includegraphics[width=7cm]{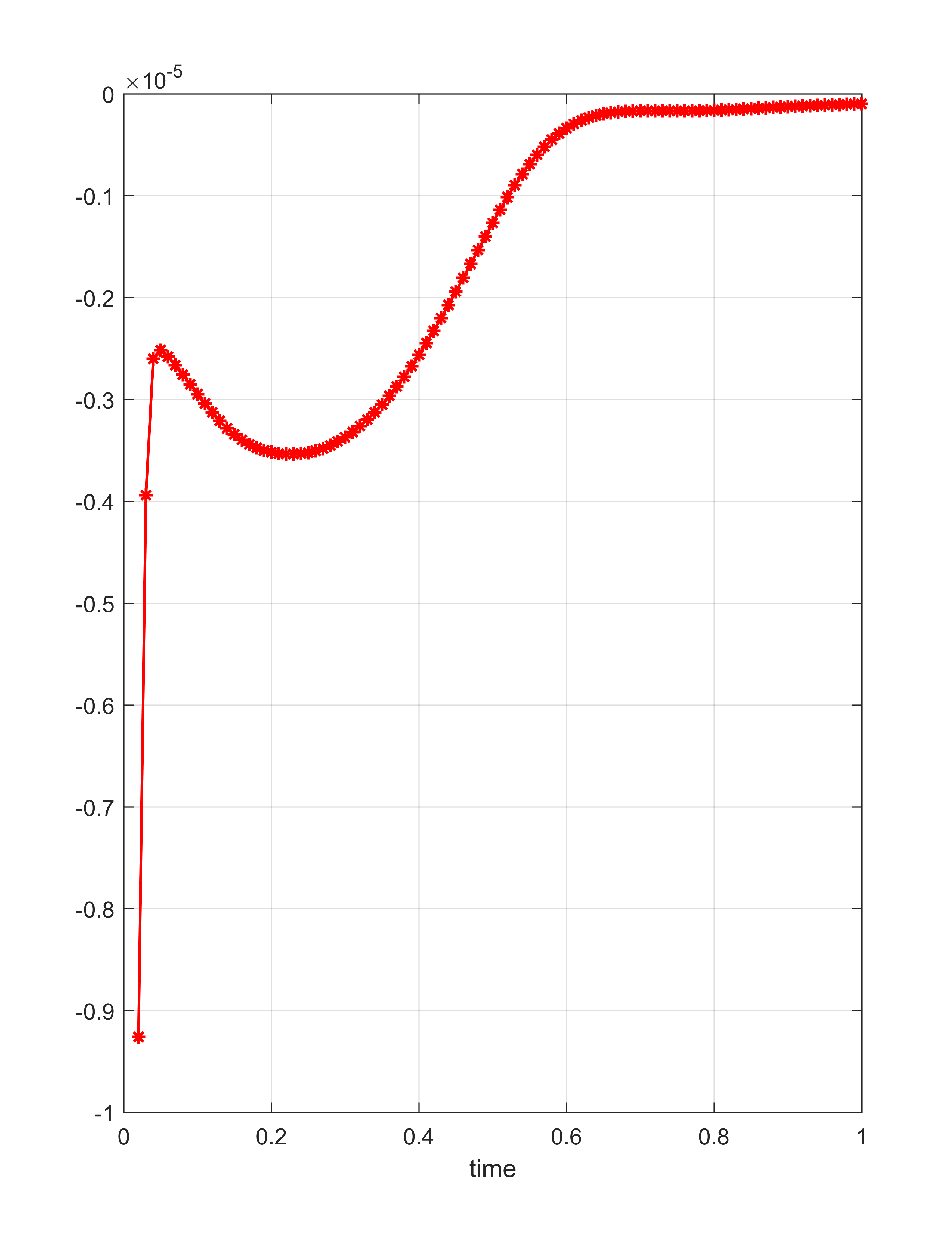}
	\caption{Discrete energy-dissipation terms (left) and their sum
		(right) in Example~\ref{ex:energy}.}
	\label{fig:energy_trend}
\end{figure}

Finally, Figure~\ref{fig:ex2fielduJ} shows the deformed 3D tetrahedral mesh, the
computed velocity field  and current density at $T=1$.
The low-conductivity material, initially occupying the upper
half of $\Omega_0$, is transported with the Lagrangian flow. The
computed current density distribution exhibits small values in the
corresponding region, qualitatively consistent with the prescribed conductivity.
\begin{figure}[!htbp]
	\centering
	\includegraphics[width=7cm]{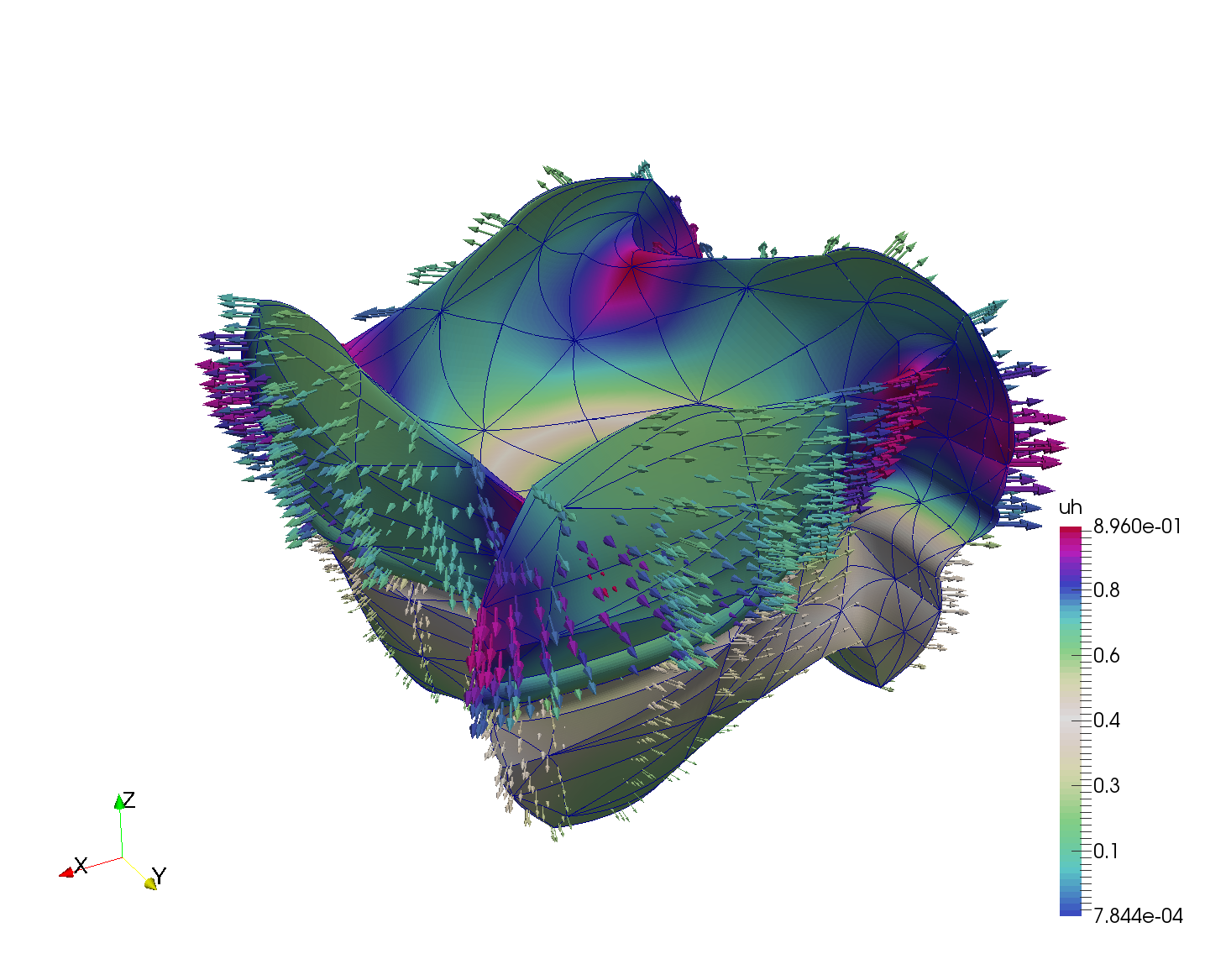}
    \includegraphics[width=7cm]{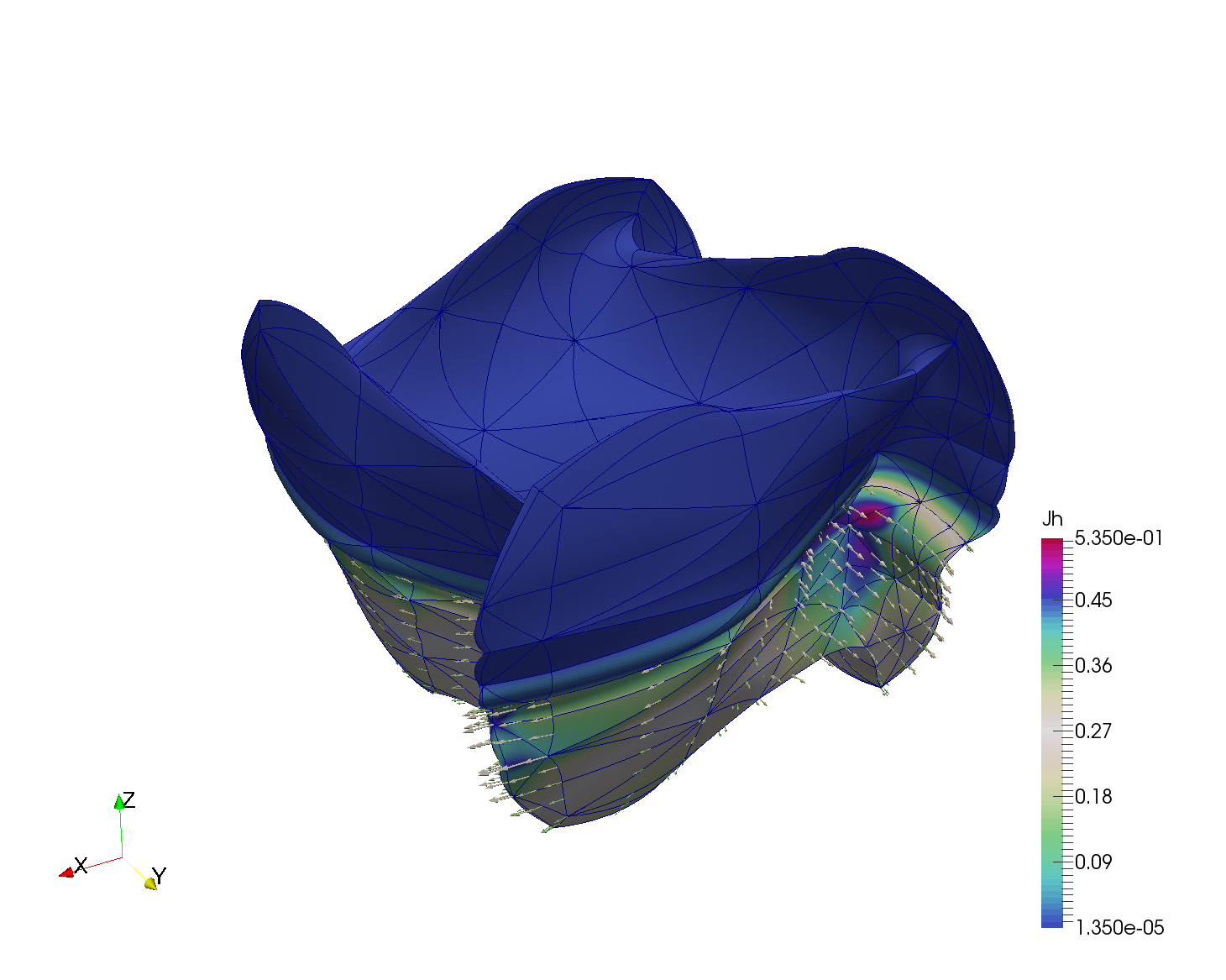}
	\caption{Deformed curved tetrahedral mesh at $T=1$,
    the velocity field (left) and current density (right) (Example~\ref{ex:energy}).}
	\label{fig:ex2fielduJ}
\end{figure}

\subsection{3D Rayleigh-Taylor instability under a constant magnetic field}\label{ex:application}
To further validate the applicability of the proposed algorithm to incompressible
inductionless MHD with moving interface. Similar to the case in \cite{dob12} (Section 8.2 therein), we simulate a 3D
magnetic Rayleigh-Taylor instability under a constant magnetic field.
The characteristic quantities used in this simulation are summarized in
Table~\ref{tab:characteristic_scales}.
\begin{table}[!htbp]
\centering
\caption{Characteristic quantities used for nondimensionalization (Example~\ref{ex:application}).}
\label{tab:characteristic_scales}
\begin{tabular}{lccc}
\hline
Characteristic quantity & Symbol & Value & Unit \\
\hline
Density
& $\rho_c$ & $1000$ & $\mathrm{kg/m^3}$ \\
Dynamic viscosity
& $\nu_c$ & $10^{-3}$ & $\mathrm{Kg/m\cdot s}$ \\
Electrical conductivity
& $\sigma_c$ & $10^{3}$ & $\mathrm{S/m}$ \\
Length
& $L_c$ & $0.1$ & $\mathrm{m}$ \\
Velocity
& $u_c$ & $1$ & $\mathrm{m/s}$ \\
Magnetic field
& $B_c$ & $1$ & $\mathrm{T}$ \\
\hline
\end{tabular}
\end{table}
Under these characteristic scales, the Reynolds and coupling numbers in this test are
\[\mathrm{Re}  =  \frac{\rho_cu_cL_c}{\nu_c}  = 10^5,\quad \alpha	= 	\frac{\sigma_cB_c^{2}L_c}{\rho_cu_c} = 0.1.\]
The applied magnetic field is $\BB = (0,1,0)^\top$, gravity $\boldsymbol{f}_u=(0,0,-0.98)^\top$
and initial domain is
\[\Omega_{0} =  [0,1] \times[0,1]\times[-1,1],\]
$\Gamma_0$ is set by $z=0$, which is the initial material interface. Free-slip boundary condition for $\Bu$ and electrical insulating boundary condition for $\BJ$ are imposed
 \[\Bu\cdot\Bn = 0, \quad  \BJ\cdot\Bn = 0  \quad \text{on } \partial\Omega.\]
Similar to the setting in \cite{dob12}, we consider a single material with initial smooth density gradient along
the $z$-direction by
\begin{equation}
\rho_0 = \frac{\rho_1+\rho_2}{2} + \frac{\rho_2-\rho_1}{\pi}\arctan(\omega z),
\end{equation}
where $\rho_1 = 0.5$, which represents the light fluid below $z=0$, $\rho_2=1.0$, which represents the heavy fluid above $z=0$. Smoothing parameter $\omega=20$. Grad-div stabilization parameter is set by $\beta=0.25$.
Moreover we apply the following initial velocity perturbation
\begin{equation}
\Bu_0 = -u_0\left(
          \begin{array}{c}
            z\sin(2\pi x)\cos(2\pi y) \\
            z\cos(2\pi x)\sin(2\pi y) \\
            \cos(2\pi x)\cos(2\pi y) \\
          \end{array}
        \right),\quad u_0 = 0.02\exp(-2\pi z^2).
\end{equation}
The initial electrical conductivity and viscosity coefficient are prescribed as
\begin{equation}
	\sigma_0(x,y,z)=
	\begin{cases}
		1.0, & 0 \le z \le1,\\
		0.1, & -1 \le z<0,
	\end{cases},\quad \nu_0(x,y,z)=
	\begin{cases}
		1.0, & 0 \le z \le1,\\
		0.25, & -1 \le z<0,
	\end{cases}	\quad (x,y,z)\in\Omega_0.
\end{equation}
For the purpose of visualization, we set $k=1$ in \eqref{eq:initial_velocity_space}.
The mesh has 49,152 elements and $h=0.1083$.
The numbers of DOFs for $\Bu_h$ and $\BJ_h$ are 212,355 and 302,592 respectively.
Time step length $\Delta t= 5 \times 10^{-3}$ and the terminal time is $T = 3.25$. The coupled linear algebraic systems \eqref{eq:fully_2st} are solved by preconditioned FGMRES solver with relative tolerance $10^{-6}$.
At the terminal time, the divergence error of $\BJ_h$ is about $2.0276 \times 10^{-11}$.

In Figure~\ref{ex3:fig1}, we firstly show the curved tetrahedral mesh near $y=0.5$ cross section, from which we can see
the mesh are extremely deformed near the material interface. Then we show the change of material interface transported by the fluid in Figure~\ref{ex3:fig3}. From those figures, we can see the heavy fluid drops in the middle, while the light fluid penetrates into the top region in the vicinity of wall.
\begin{figure}[!htbp]
	\centering
	\includegraphics[width=8cm]{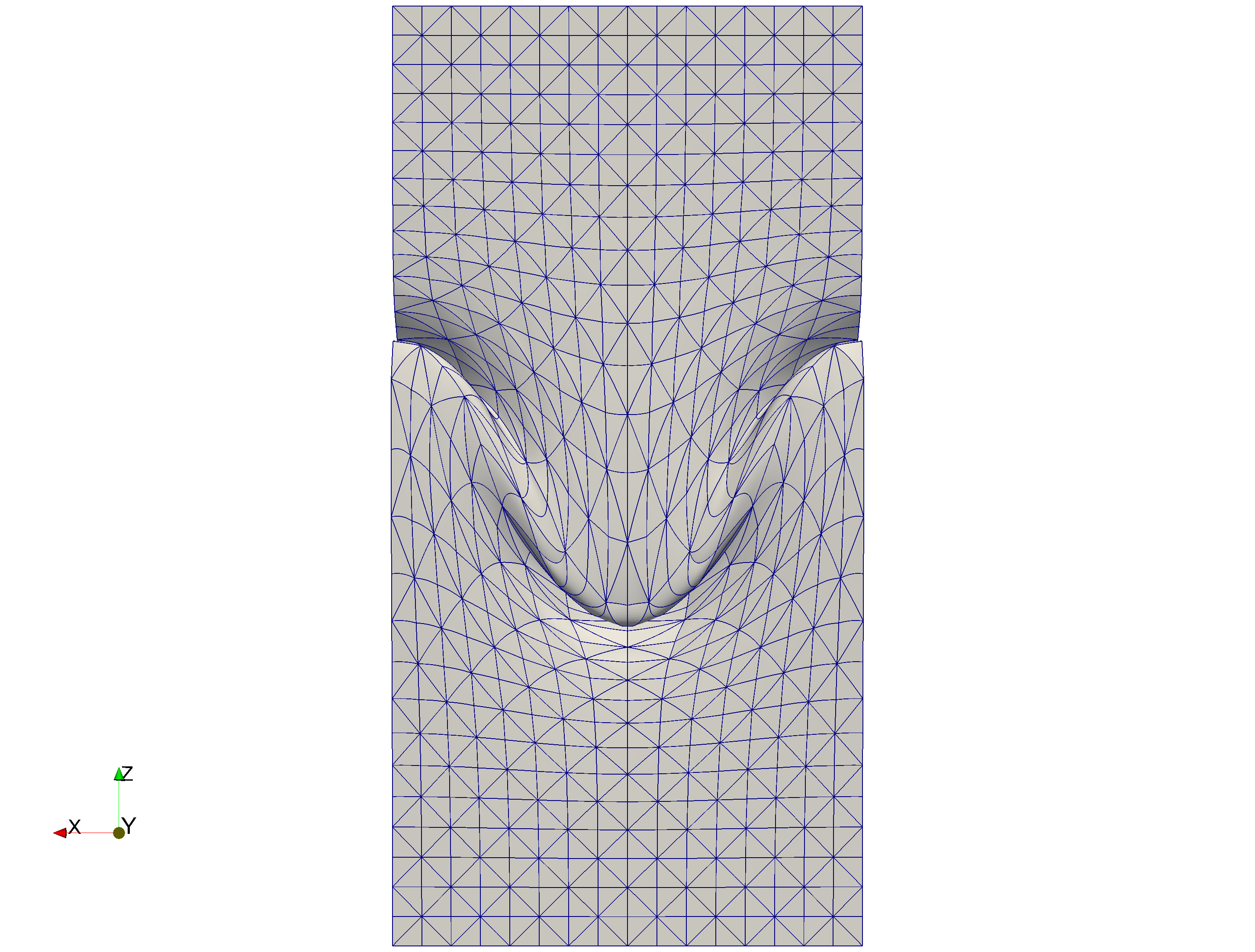}
	\includegraphics[width=6cm]{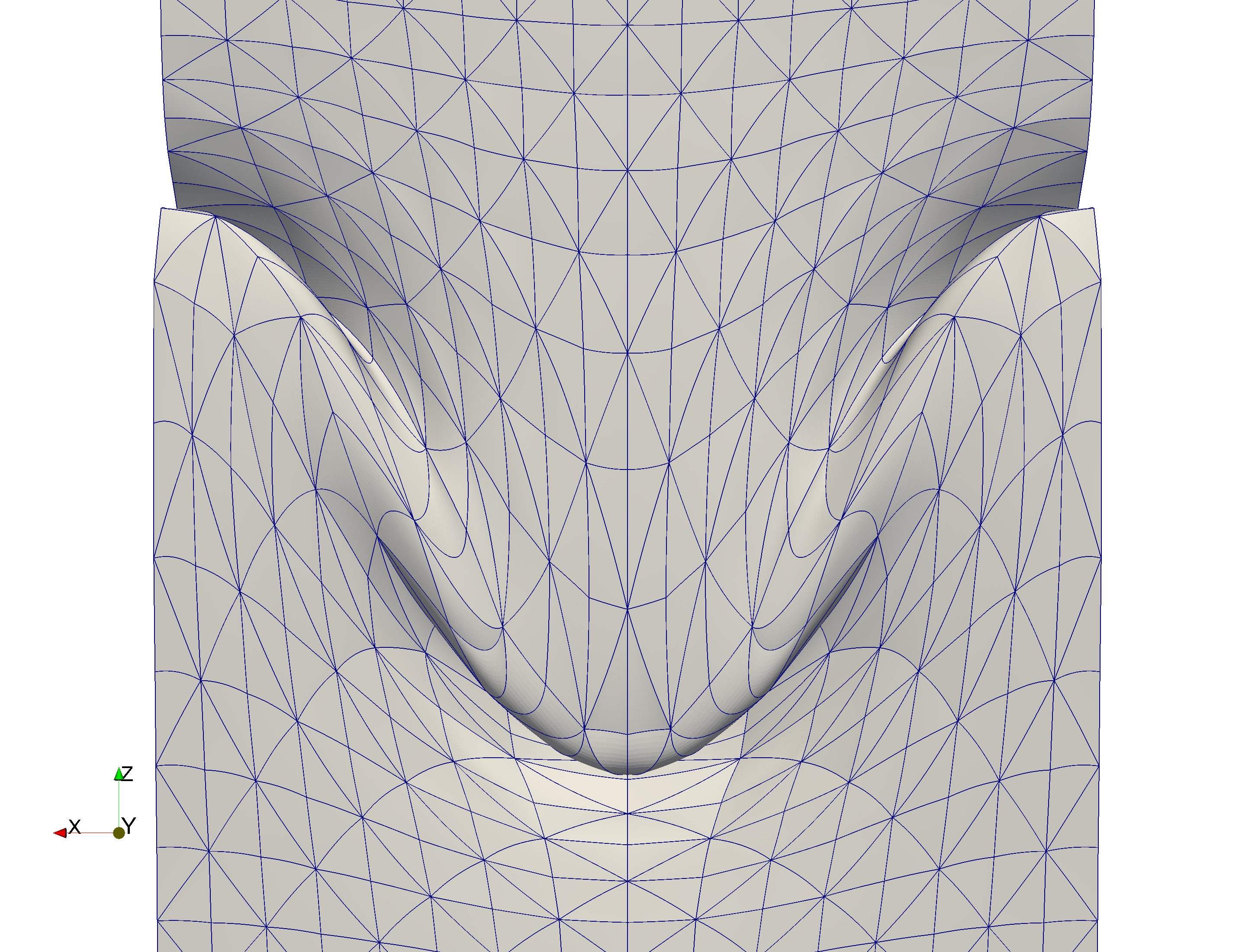}
	\caption{The curved meshes in the vicinity of $y=0.5$ cross section at time $T=3.25$.
  The right one is the zoom in of left panel (Example~\ref{ex:application}).}
  \label{ex3:fig1}
\end{figure}

\begin{figure}[!htbp]
	\centering
	\includegraphics[width=6cm]{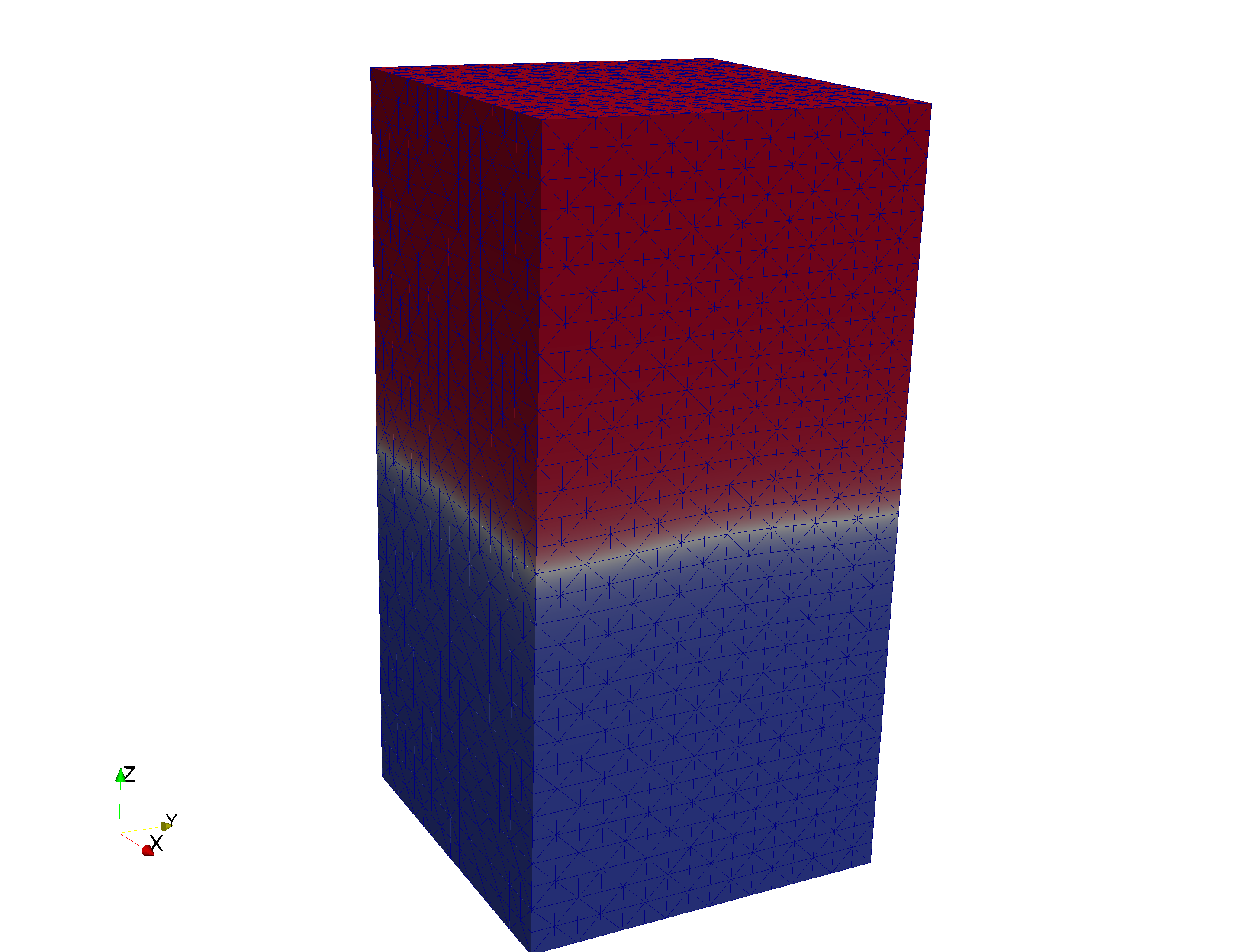}
	\includegraphics[width=6cm]{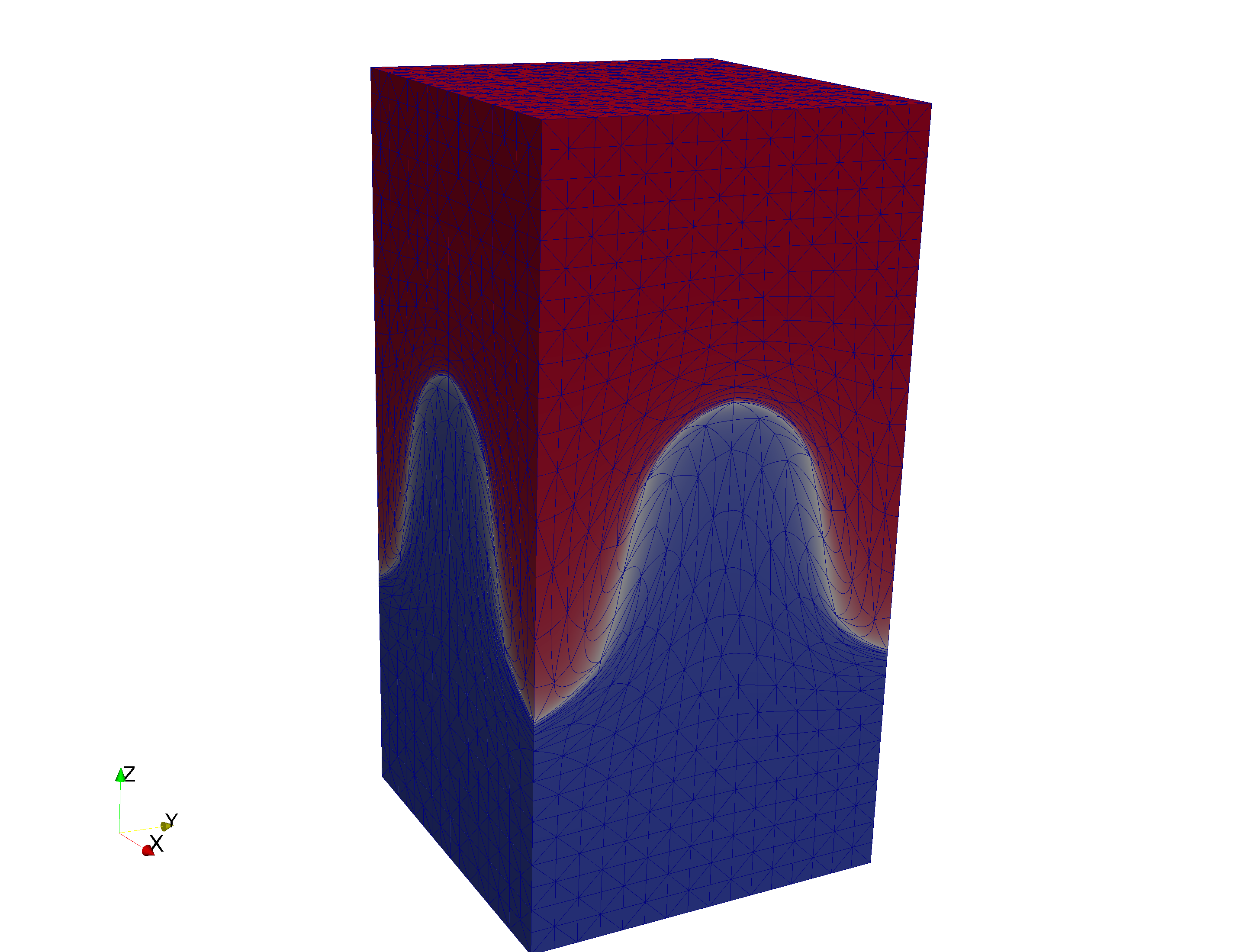}
	\caption{The material interface distribution at
     $t=0.45$ (left) and $t=3.25$ (right). The top is heavy fluid and the bottom is
     light fluid. Note that the initial interface is $z=0$ (Example~\ref{ex:application}).}
     \label{ex3:fig3}
\end{figure}

Figure~\ref{ex3:fig2} shows that the acceleration of the system under the gravity and the dissipation terms also grows with respect to time. Then Figure~\ref{ex3:fig4} and Figure~\ref{ex3:fig5} show the magnitude and vector distribution of $\Bu_h$
in the vicinity of $x=0.5$ cross section at $t=0.45$ and $t=3.25$,
from which we again observe how the heavy fluid descends and the light fluid rises.
Finally, Figure~\ref{ex3:fig6} plot the distribution of the magnitude of $\BJ_h$ in the vicinity of $x=0.5$ cross section
and Figure~\ref{ex3:fig7} gives the distribution of current density vector in the vicinity of $z=0.05$ cross section.
From them, we can see complex and symmetric current density distribution.

\begin{figure}[!htbp]
	\centering
	\includegraphics[width=6cm]{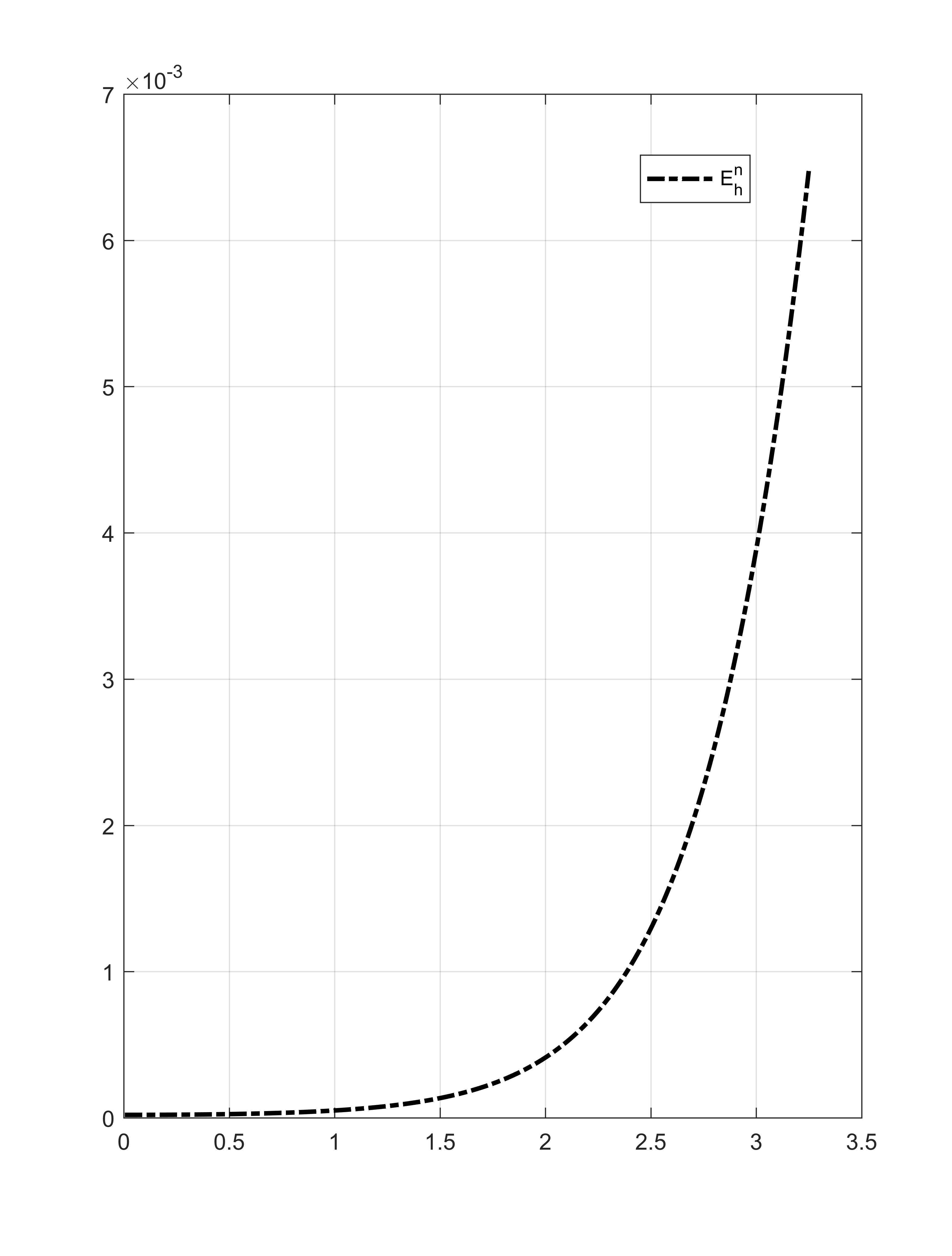}
	\includegraphics[width=6cm]{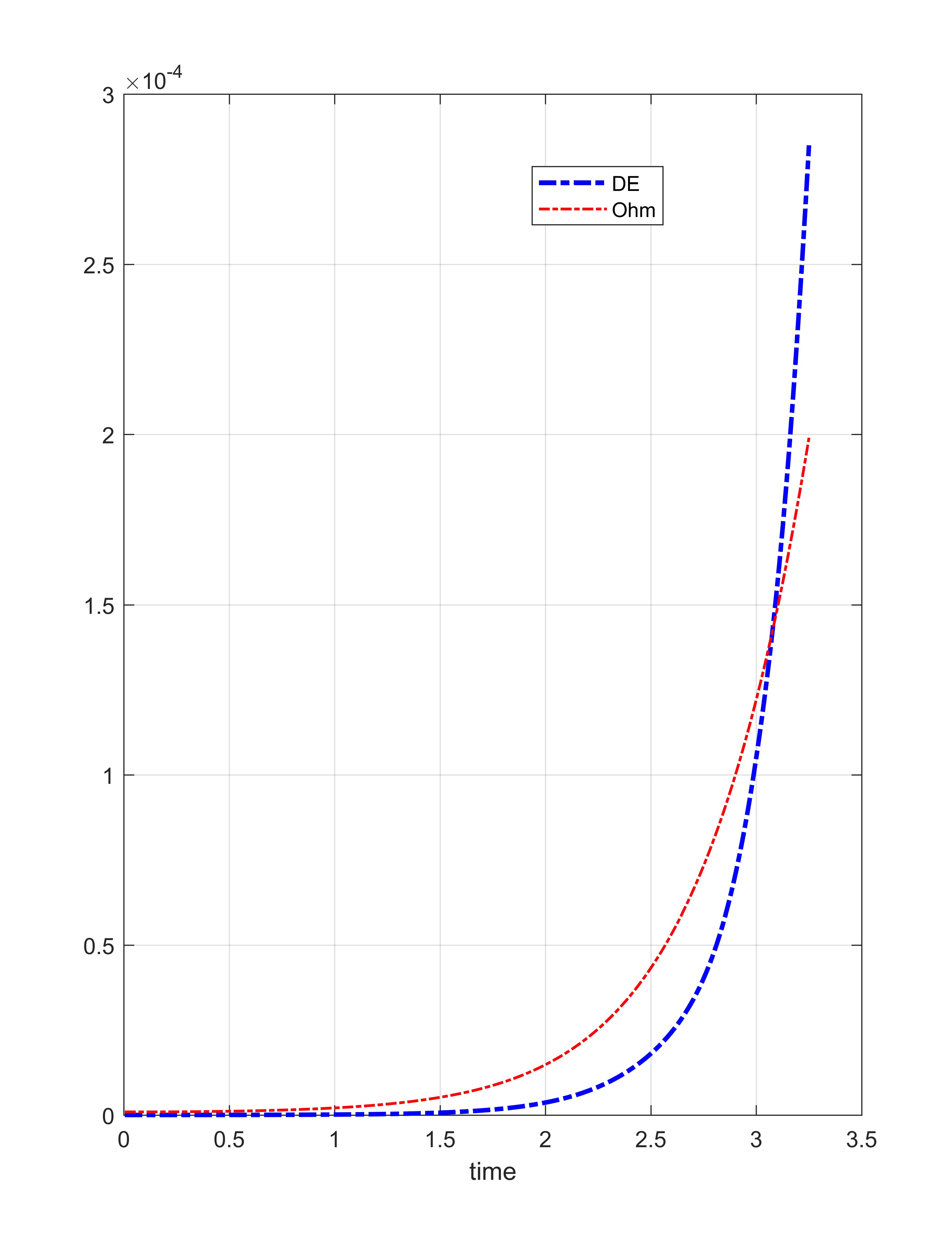}
	\caption{The evolution of $\Ce_h^n$ (left) and
    $\mathcal{D}_{E}^{n},\mathrm{Ohm}^{n}$ (right) (Example~\ref{ex:application}).}
    \label{ex3:fig2}
\end{figure}

\begin{figure}[!htbp]
	\centering
	\includegraphics[width=6cm]{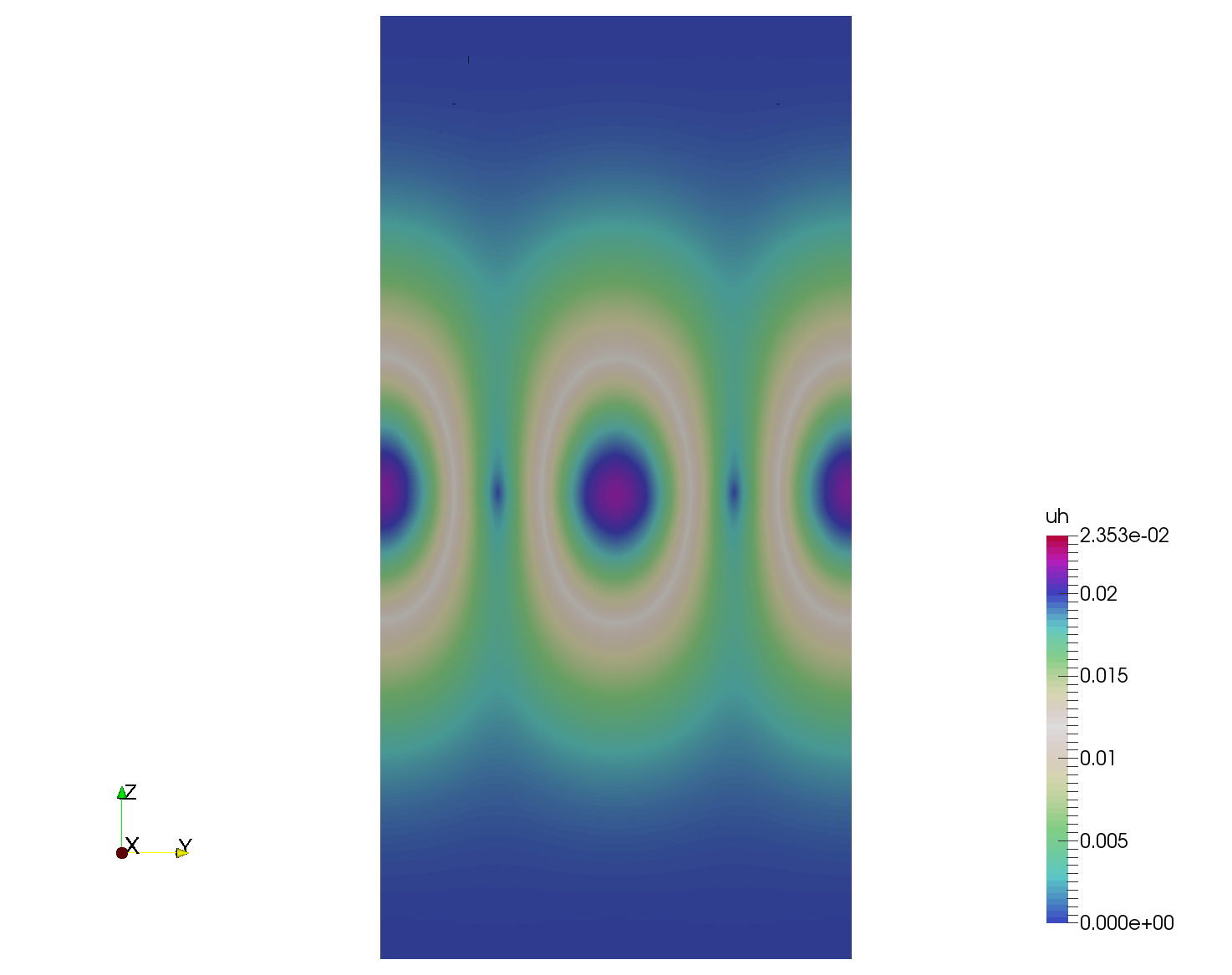}
	\includegraphics[width=6cm]{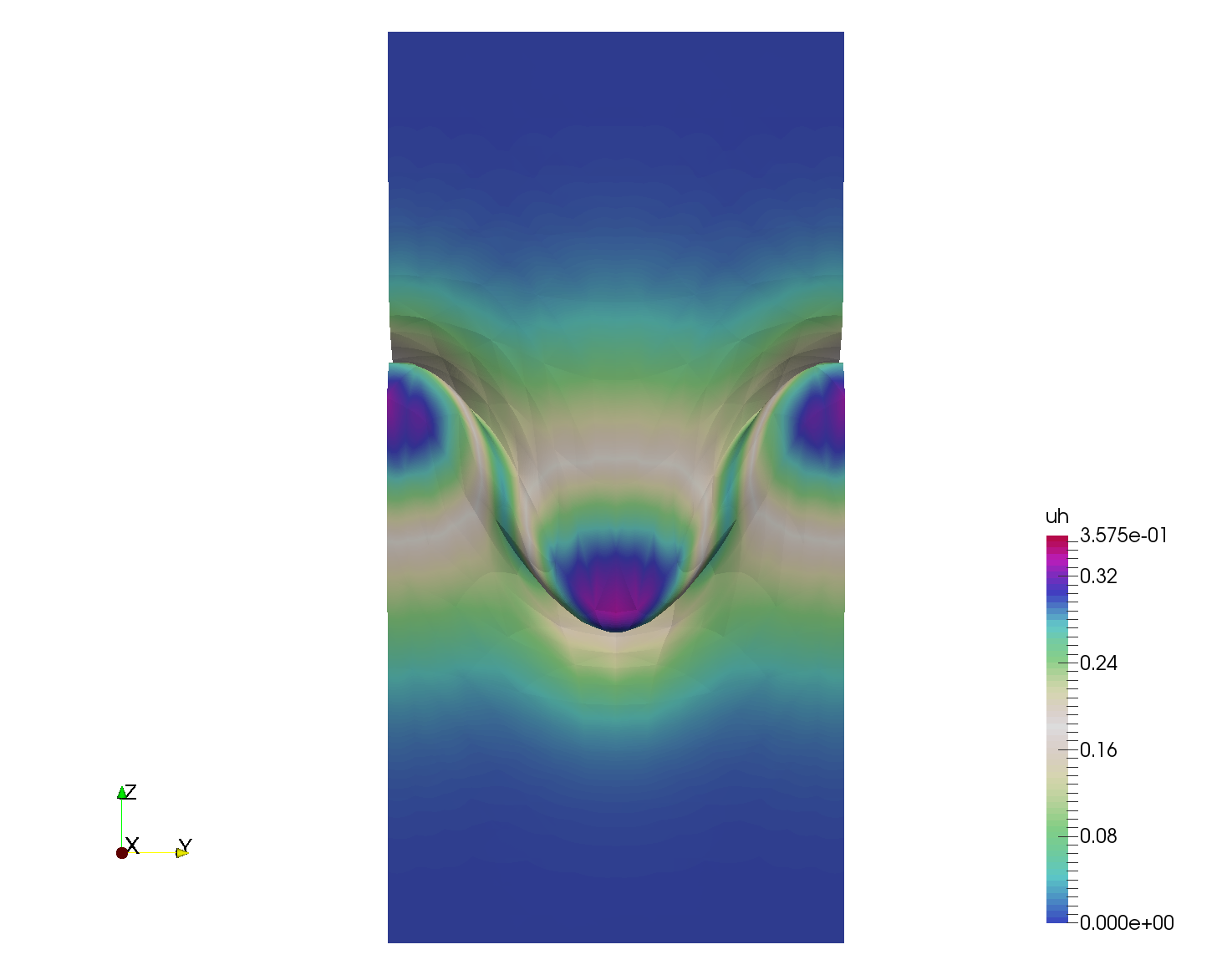}
	\caption{Distribution of the magnitude of $\Bu_h$ in the vicinity of $x=0.5$ cross section.
    The left panel is at $t=0.45$ and the right one is at $t=3.25$ (Example~\ref{ex:application}).}
    \label{ex3:fig4}
\end{figure}

\begin{figure}[!htbp]
	\centering
	\includegraphics[width=6cm]{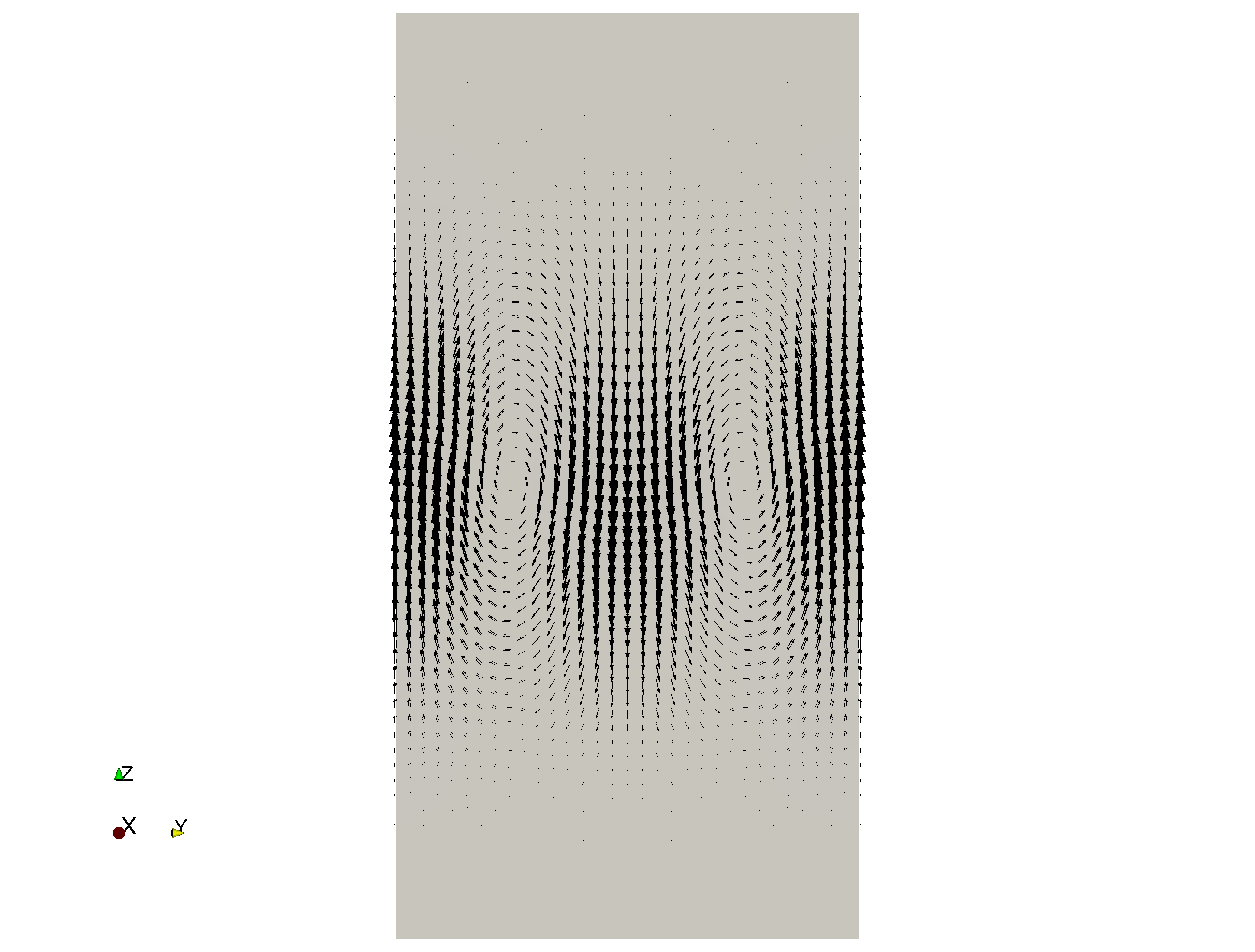}
	\includegraphics[width=6cm]{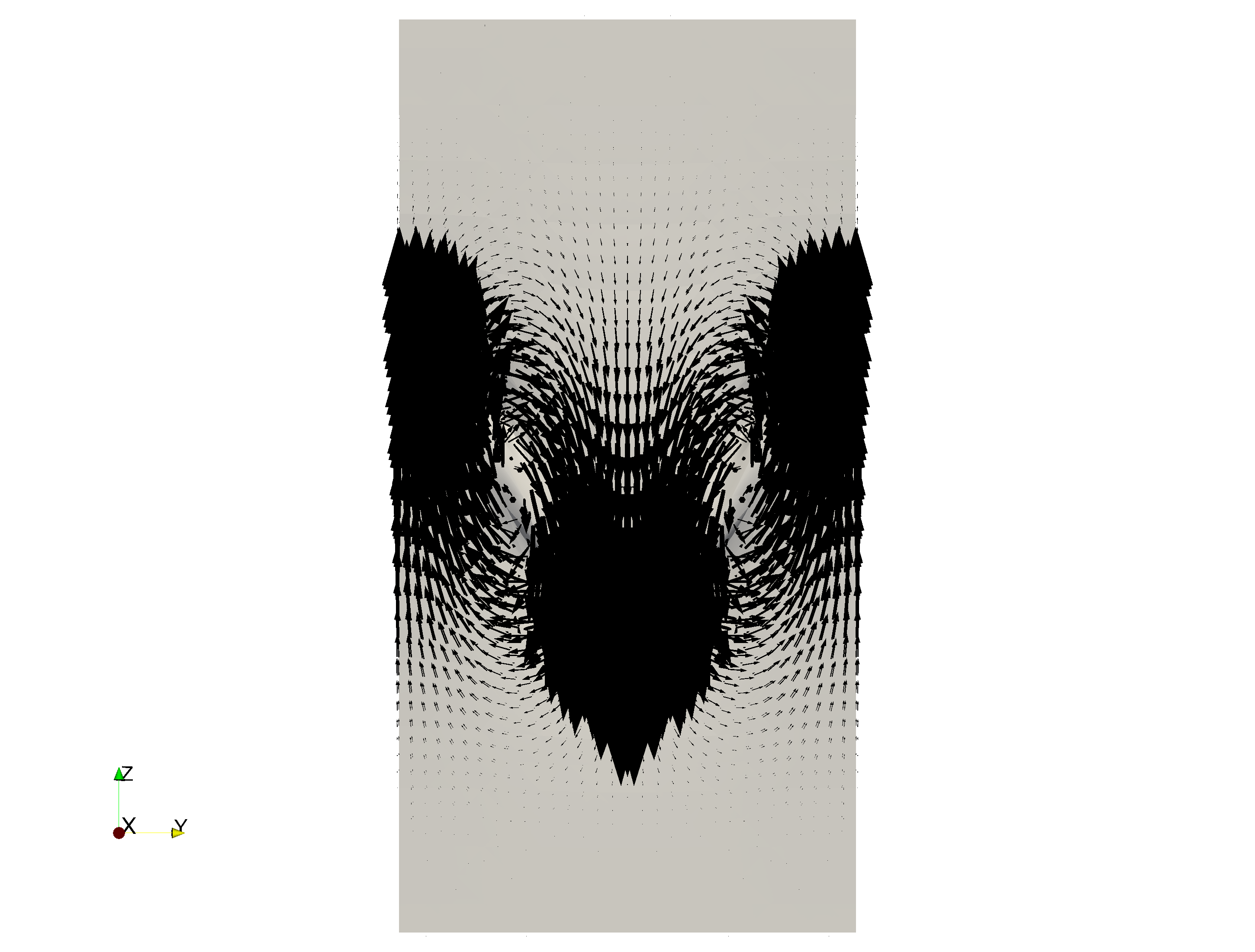}
	\caption{Distribution of velocity vector in the vicinity of $x=0.5$ cross section.
    The left panel is at $t=0.45$ (scaled by a factor of 5)
    and the right one is at $t=3.25$ (Example~\ref{ex:application}).}
    \label{ex3:fig5}
\end{figure}

\begin{figure}[!htbp]
	\centering
	\includegraphics[width=6cm]{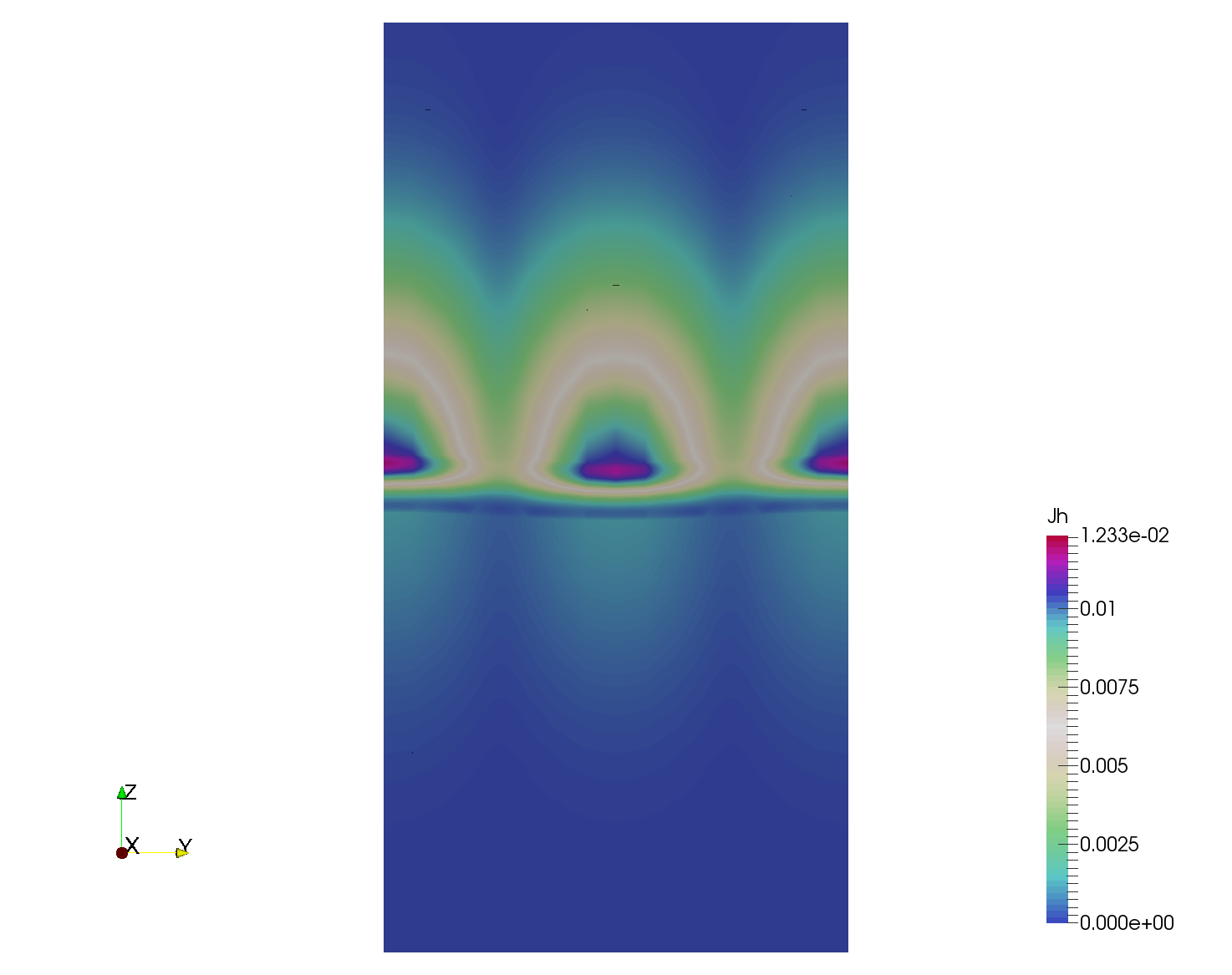}
	\includegraphics[width=6cm]{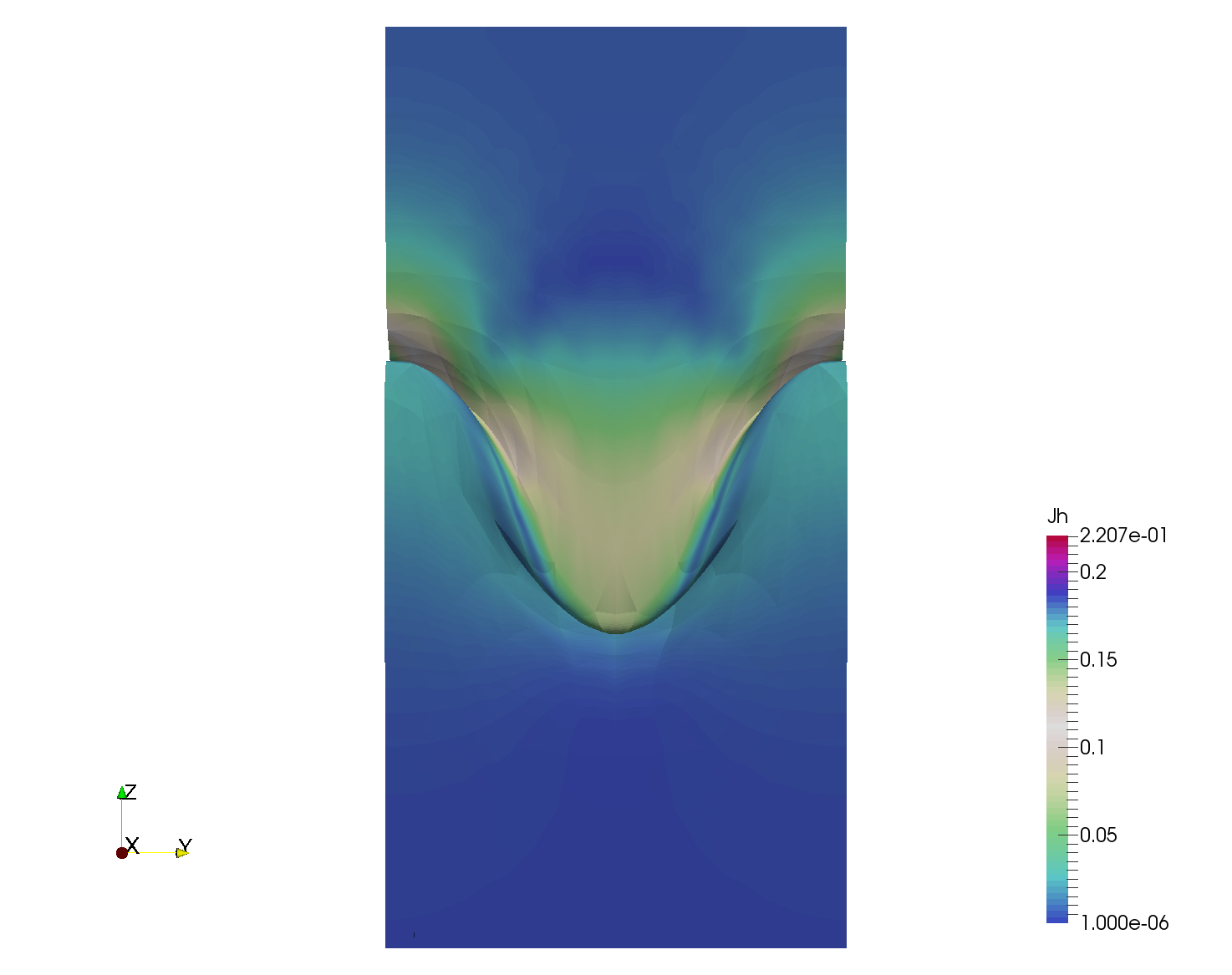}
	\caption{Distribution of the magnitude of $\BJ_h$ in the vicinity of $x=0.5$ cross section.
    The left panel is at $t=0.45$ and the right one is at $t=3.25$ (Example~\ref{ex:application}).}
    \label{ex3:fig6}
\end{figure}

\begin{figure}[!htbp]
	\centering
	\includegraphics[width=7cm]{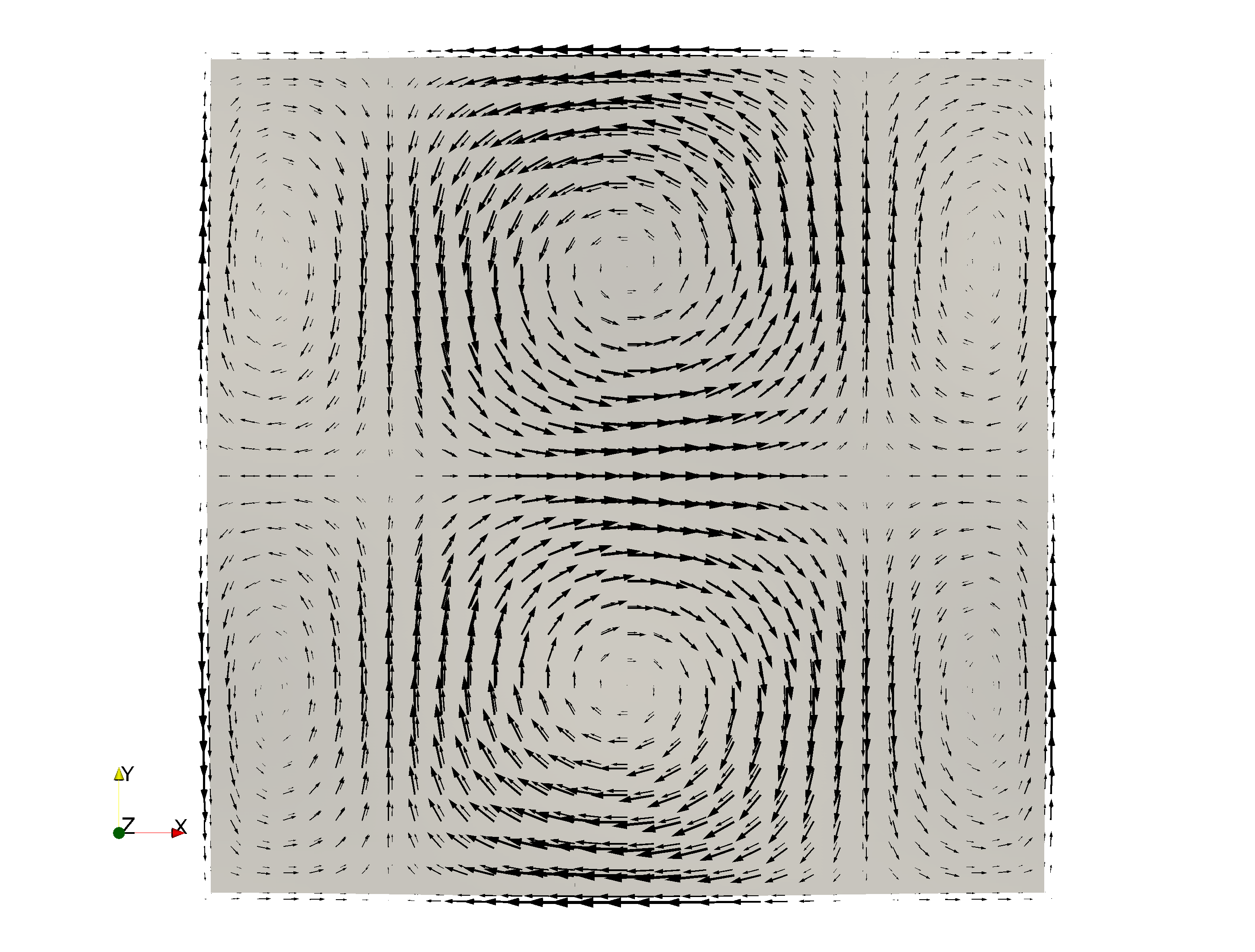}
	\includegraphics[width=7cm]{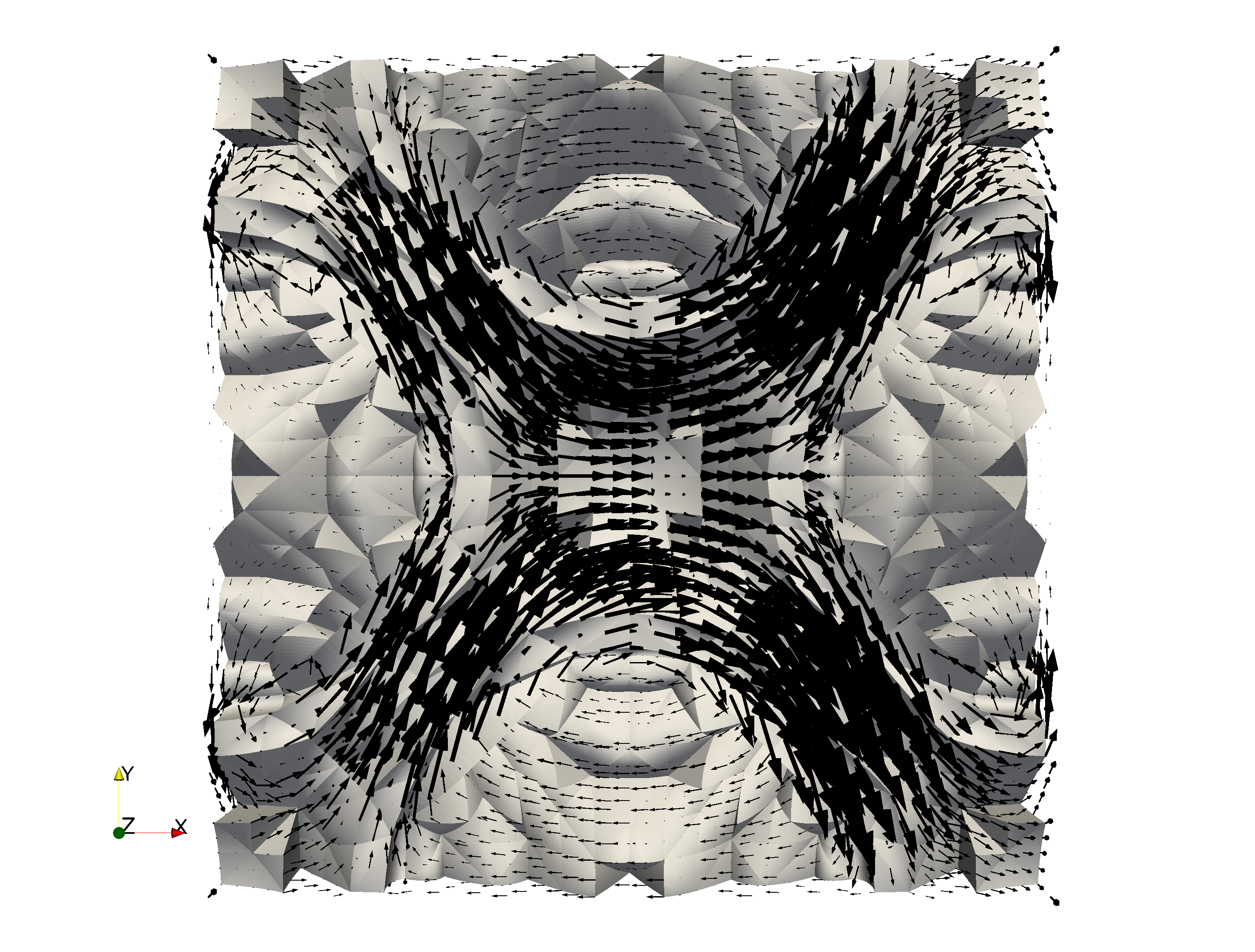}
	\caption{Distribution of current density vector in the vicinity of $z=0.05$ cross section.
    The left panel is at $t=0.45$ (scaled by a factor of 5)
    and the right one is at $t=3.25$ (Example~\ref{ex:application}).}
    \label{ex3:fig7}
\end{figure}

\section{Conclusion}\label{sec:conclusion}
In this work, we developed a high-order Lagrangian mixed finite element
method for three-dimensional variable-density incompressible inductionless
MHD with possible moving interface.
The resulting scheme preserves the charge-conservation constraint exactly
on evolving curvilinear meshes and satisfies a discrete energy identity; in the absence of external force, it is
unconditionally energy stable.
These properties are achieved through a strong mass conservation formulation for the density field together with
high-order parametric $\BH(\Div)$-conforming discretization of current density.
A series of numerical experiments demonstrate the expected convergence rates for smooth solutions,
the discrete energy-dissipation property and applicability of the method when moving interface is present.

\input{refs}

\end{document}

%% file: macros.tex
\providecommand{\Div}{\operatorname{div}}          % Divergence
\providecommand*{\Dist}[2]{\operatorname{dist}({#1};{#2})}   % distance
\providecommand*{\Dist}[2]{\Dist{#1}{#2}}
\newcommand{\Bf}{{\boldsymbol{f}}}

\newcommand{\Bn}{{\boldsymbol{n}}}

\newcommand{\Bu}{{\boldsymbol{u}}}
\newcommand{\Bv}{{\boldsymbol{v}}}

\newcommand{\Bx}{{\boldsymbol{x}}}

\newcommand{\Bz}{{\boldsymbol{z}}}

\newcommand{\BB}{{\boldsymbol{B}}}

\newcommand{\BD}{{\boldsymbol{D}}}

\newcommand{\BF}{{\boldsymbol{F}}}

\newcommand{\BH}{{\boldsymbol{H}}}

\newcommand{\BJ}{{\boldsymbol{J}}}

\newcommand{\BM}{{\boldsymbol{M}}}

\newcommand{\BP}{{\boldsymbol{P}}}

\newcommand{\BV}{{\boldsymbol{V}}}
\newcommand{\BW}{{\boldsymbol{W}}}

\newcommand{\Ce}{\mathcal{E}}

\newcommand{\Cj}{\mathcal{J}}

\newcommand{\Ct}{\mathcal{T}}

\newcommand{\bbI}{\mathbb{I}}

\newcommand{\bbR}{\mathbb{R}}

\newcommand{\be}{\begin{eqnarray}}
\newcommand{\ee}{\end{eqnarray}}

\newcommand{\ben}{\begin{eqnarray*}}
\newcommand{\een}{\end{eqnarray*}}

%% file: refs.tex
\bibliographystyle{amsplain}